\documentclass[a4paper,12pt]{amsart}

\usepackage{amsmath}
\usepackage{amssymb}
\usepackage{amsthm}
\usepackage{todonotes}
\usepackage{enumerate}   
\usepackage{cite}
\DeclareFontFamily{U}{mathx}{\hyphenchar\font45}
\DeclareFontShape{U}{mathx}{m}{n}{
      <5> <6> <7> <8> <9> <10>
      <10.95> <12> <14.4> <17.28> <20.74> <24.88>
      mathx10
      }{}
\DeclareSymbolFont{mathx}{U}{mathx}{m}{n}
\DeclareFontSubstitution{U}{mathx}{m}{n}
\DeclareMathAccent{\widecheck}{0}{mathx}{"71}

\makeatletter
\newcommand\xleftrightarrow[2][]{%
  \ext@arrow 9999{\longleftrightarrowfill@}{#1}{#2}}
\newcommand\longleftrightarrowfill@{%
  \arrowfill@\leftarrow\relbar\rightarrow}
\makeatother

\theoremstyle{plain}
\newtheorem{thm}{Theorem}[section]
\newtheorem{cor}{Corollary}[thm]
\newtheorem{lem}[thm]{Lemma}
\newtheorem{prop}[thm]{Proposition}

\theoremstyle{definition}

\newtheorem{exmp}{Example}[section]

\theoremstyle{remark}
\newtheorem*{rem}{Remark}

\newtheorem*{notation}{Notation}

\newcommand\inner[2]{\left\langle #1, #2 \right\rangle}

\newcommand{\tco}{\mathcal{T}}
\newcommand{\bo}{\mathcal{L}(L^2)}
\newcommand{\SC}{\mathcal{T}}
\newcommand{\beauty}{\mathcal{B}}
\newcommand{\beast}{\mathcal{B}'}

\newcommand{\R}{\mathbb{R}}
\newcommand{\Rd}{\mathbb{R}^d}
\newcommand{\Rdd}{\mathbb{R}^{2d}}
\newcommand{\Z}{\mathbb{Z}}
\newcommand{\N}{\mathbb{N}}

\newcommand{\HS}{\mathcal{HS}}
\newcommand{\F}{\mathcal{F}}
\newcommand{\tr}{\mathrm{tr}}
\newcommand{\weyl}{a}

\begin{document}
\pagestyle{plain}
\title{Quantum harmonic analysis on lattices and Gabor multipliers}
\author{Eirik Skrettingland} 
\address{Department of Mathematics\\ NTNU Norwegian University of Science and
Technology\\ NO–7491 Trondheim\\Norway}
\email{eirik.skrettingland@ntnu.no}
\begin{abstract}
 We develop a theory of quantum harmonic analysis on lattices in $\Rdd$. Convolutions of a sequence with an operator and of two operators are defined over a lattice, and using corresponding Fourier transforms of sequences and operators we develop a version of harmonic analysis for these objects. We prove analogues of results from classical harmonic analysis and the quantum harmonic analysis of Werner, including Tauberian theorems and a Wiener division lemma. Gabor multipliers from time-frequency analysis are described as convolutions in this setting. The quantum harmonic analysis is thus a conceptual framework for the study of Gabor multipliers, and several of the results include results on Gabor multipliers as special cases.   
\end{abstract}
\maketitle \pagestyle{myheadings} \markboth{E. Skrettingland}{Quantum harmonic analysis on lattices}
\thispagestyle{empty}

\section{Introduction}
In time-frequency analysis, one studies a signal $\psi \in L^2(\Rd)$ by considering various time-frequency representations of $\psi$. An important class of time-frequency representations is obtained by fixing $\varphi\in L^2(\Rd)$ and considering the \textit{short-time Fourier transform} $V_\varphi \psi$ of $\psi$ with window $\varphi$, which is the function on the time-frequency plane $\Rdd$ given by
\begin{equation*}
  V_{\varphi}\psi(z)=\inner{\psi}{\pi(z)\varphi}_{L^2} \quad \text{ for } z\in \Rdd,
\end{equation*}
where $\pi(z):L^2(\Rd)\to L^2(\Rd)$ is the \textit{time-frequency shift} given by $\pi(z)\varphi(t)=e^{2\pi i \omega \cdot t}\varphi(t-x)$ for $z=(x,\omega).$ The intuition is that $V_\varphi \psi(z)$ carries information about the components of the signal $\psi$ with frequency $\omega$ at time $x$. 

A question going back to von Neumann \cite{vonNeumann:1955} and Gabor \cite{Gabor:1946} is the validity of reconstruction formulas of the form 
\begin{equation} \label{eq:intro1}
  \psi = \sum_{\lambda \in \Lambda} V_\varphi \psi(\lambda) \pi(\lambda) \xi \text{ for any } \psi\in L^2(\Rd),
\end{equation}
where $\Lambda=A\Z^{2d}$ for $A\in GL(2d,\R)$ is a lattice in $\Rdd$ and $\varphi,\xi \in L^2(\Rd)$. It is known that \eqref{eq:intro1} is indeed true for certain windows $\varphi,\xi$ and lattices $\Lambda$, and such formulas naturally lead to the concept of \textit{Gabor multipliers}. If $\varphi,\xi \in L^2(\Rd)$ and $m=\{m(\lambda)\}_{\lambda \in \Lambda}$ is a sequence of complex numbers, we define the Gabor multiplier $\mathcal{G}_{m}^{\varphi,\xi}:L^2(\Rd) \to L^2(\Rd)$ by
\begin{equation*}
  \mathcal{G}_{m}^{\varphi,\xi}(\psi)=\sum_{\lambda \in \Lambda} m(\lambda)V_\varphi \psi(\lambda) \pi(\lambda) \xi.
\end{equation*}
Compared to \eqref{eq:intro1} we see that $\mathcal{G}_{m}^{\varphi,\xi}$ modifies the time-frequency content of $\psi$ in a simple way, namely by multiplying the samples of its time-frequency representation with a mask $m$. Gabor multipliers have been studied in the mathematics literature by  \cite{Grochenig:2013,Feichtinger:2002,Feichtinger:2003,Grochenig:2011,Dorfler:2010,Benedetto:2006,Feichtinger:1998a,Cordero:2003} among others, and also in more application-oriented contributions \cite{Balazs:2010,Taubock:2019,Rajbamshi:2019}.

Gabor multipliers are the discrete analogues of the much-studied localization operators \cite{Daubechies:1990,Cordero:2003,Bayer:2014,Grochenig:2011toft}. In \cite{Luef:2018c} we showed that the quantum harmonic analysis developed by Werner and coauthors \cite{Werner:1984,Kiukas:2012} provides a conceptual framework for localization operators, leading to new results and interesting reinterpretations of older results on localization operators. The goal of this paper is therefore to develop a version of quantum harmonic analysis for lattices to provide a similar conceptual framework for Gabor multipliers. Hence we continue the line of research into applications of quantum harmonic analysis from \cite{Luef:2018c,Luef:2018b,Luef:2018}.

With this aim we introduce two convolutions of operators and sequences in Section \ref{sec:convolutions}. Following \cite{Werner:1984,Kozek:1992a,Feichtinger:1998} we first define the translation of an operator $S$ on $L^2(\Rd)$ by $\lambda\in \Lambda$ to be the operator $$\alpha_\lambda(S)=\pi(\lambda)S\pi(\lambda)^*.$$ If $c\in \ell^1(\Lambda)$ and $S$ is a trace class operator on $L^2(\Rd)$, the convolution $c \star_\Lambda S$ is defined to be the \textit{operator}
\begin{equation*}
  c \star_\Lambda S=\sum_{\lambda \in \Lambda} c(\lambda)\alpha_\lambda(S).
\end{equation*}
Gabor multipliers are then given by convolutions
\begin{equation*}
   \mathcal{G}_{m}^{\varphi,\xi}=m\star_\Lambda (\xi \otimes \varphi),
\end{equation*}
where $\xi \otimes \varphi$ is the  rank-one operator $\xi \otimes \varphi(\psi)=\inner{\psi}{\varphi}_{L^2} \xi$.
Furthermore, we define the convolution $S\star_\Lambda T$ of two trace class operators $S$ and $T$ to be the \textit{sequence} over $\Lambda$ given by 
\begin{equation*}
  S\star_\Lambda T(\lambda)=\tr(S\alpha_\lambda (\check{T})),
\end{equation*}
where $\check{T}=PTP$ with $P$ the parity operator $P\psi(t)=\psi(-t)$ for $\psi \in L^2(\Rd)$. In Section \ref{sec:convolutions} we investigate the commutativity and associativity of these convolutions, extend their domains and in Proposition \ref{prop:youngschatten} we establish a version of Young's inequality for convolutions of operators and sequences. 

An important tool throughout the paper is a Banach space $\beauty$ of trace class operators, consisting of operators with Weyl symbol in the so-called Feichtinger algebra \cite{Feichtinger:1981}. The use of $\beauty$ allows us to obtain continuity results for the convolutions with respect to $\ell^p(\Lambda)$ and Schatten-$p$ classes -- an important example is Proposition \ref{prop:convwelldefined} which states that 
\begin{equation*}
  \|S\star_\Lambda T\|_{\ell^1(\Lambda)}\lesssim \|S\|_{\beauty} \|T\|_{\tco}
\end{equation*}
for $S\in \beauty$ and trace class $T$, where $\|\cdot\|_\tco$ is the trace class norm. While there are other classes of operators that would ensure that $S\star_\Lambda T\in \ell^1(\Lambda)$, see for instance the Schwartz operators \cite{Keyl:2015}, $\beauty$ has the advantage of being a Banach space, hence allowing the use of tools such as Banach space adjoints. The space $\beauty$ has previously been studied by \cite{Feichtinger:1998,Dorfler:2010,Feichtinger:2018} among others.

To complement the convolutions, we introduce Fourier transforms of sequences and operators in Section \ref{sec:fouriertransforms}. For a sequence  $c\in \ell^1(\Lambda)$ we use its symplectic Fourier series 
\begin{equation*}
  \F_\sigma^\Lambda(c)(z)=\sum_{\lambda \in \Lambda} c(\lambda)e^{2\pi i \sigma(\lambda,z)} \quad \text{ for } z\in \Rdd,
\end{equation*}
where $\sigma(z,z')=\omega\cdot x'-x\cdot \omega'$ for $z=(x,\omega),z'=(x',\omega').$ As a Fourier transform for trace class operators $S$ we use the Fourier-Wigner transform
\begin{equation*}
 \F_W(S)(z)= e^{-\pi i x\cdot \omega} \tr(\pi(-z)S) \quad \text{ for } z=(x,\omega)\in \Rdd.
\end{equation*}
 Equipped with both convolutions and Fourier transforms, we naturally ask whether the Fourier transforms turn convolutions into products. We show in Theorem \ref{thm:orthogonality} for $z\in \Rdd$ that
\begin{equation} \label{eq:introFs}
		\F_{\sigma}^{\Lambda}(S\star_\Lambda T)(z)=\frac{1}{|\Lambda|} \sum_{\lambda^\circ\in \Lambda^\circ}F_W(S)(z+\lambda^\circ)\F_W(T)(z+\lambda^\circ),
	\end{equation}
	where $\Lambda^\circ$ is the adjoint lattice of $\Lambda$ defined in Section \ref{sec:fouriertransforms}, and in Propositions \ref{prop:spreadinggm} and \ref{prop:spreadinggm2} we show that
	\begin{equation} \label{eq:introFW}
  \F_W(c\star_\Lambda S)(z)=\F_\sigma^\Lambda(c)(z)\F_W(S)(z).
\end{equation}
These results include as special cases the so-called fundamental identity of Gabor analysis \cite{Rieffel:1988,Feichtinger:2006,Tolimieri:1994,Janssen:1995} and results on the spreading function of Gabor multipliers due to \cite{Dorfler:2010}. Equations \eqref{eq:introFs} and \eqref{eq:introFW} hold for general classes of operators and sequences, and we take care to give a precise interpretation of the objects and equalities in all cases. 

A fruitful approach to Gabor multipliers due to Feichtinger \cite{Feichtinger:2002} is to consider the so-called Kohn-Nirenberg symbol of operators. The Kohn-Nirenberg symbol of an operator $S$ on $L^2(\Rd)$ is a function on $\Rdd$, and Feichtinger used this to reduce questions about Gabor multipliers in the Hilbert Schmidt operators to questions about functions in $L^2(\Rdd)$. This approach has later been used in other papers on Gabor multipliers \cite{Dorfler:2010,Benedetto:2006,Feichtinger:2003}. As Gabor multipliers are examples of convolutions, we show in Section \ref{sec:riesz} that this approach can be generalized and phrased in terms of our quantum harmonic analysis, and that one of the main results of \cite{Feichtinger:2002} finds a natural interpretation as a Wiener's lemma in our setting -- see Theorem \ref{thm:biorthogonal}, Corollary \ref{cor:banachisomorphism} and the remarks following the corollary. 

In Section \ref{sec:tauberian} we show the extension of some deeper results of harmonic analysis on $\Rd$ to our setting. We obtain an analogue of Wiener's classical Tauberian theorem in Theorem \ref{thm:bigtauberian}, similar to the results of Werner and coauthors \cite{Werner:1984,Kiukas:2012} in the continuous setting. As an example we have the following equivalent statements for $S\in \beauty:$
\begin{enumerate}[(i)]
			\item The set of zeros of $\F_\sigma^\Lambda(S\star_\Lambda \check{S}^*)$ contains no open subsets in $\Rdd/\Lambda^\circ$.
			\item If $c\star_\Lambda S=0$ for $c\in \ell^1(\Lambda)$, then $c=0$.
			\item $\beast \star_\Lambda S$ is weak*-dense in $\ell^\infty(\Lambda)$.
\end{enumerate}
These results are related to earlier investigations of Gabor multipliers by Feichtinger  \cite{Feichtinger:2002}. In particular, he showed that if $S=\xi\otimes \varphi$ is a rank-one operator and  $\F_\sigma^\Lambda(S\star_\Lambda \check{S}^*)$ has \textit{no} zeros, then any $m\in \ell^\infty(\Lambda)$ can be recovered from the Gabor multiplier $\mathcal{G}_{m}^{\varphi,\xi}$. Since Gabor multipliers are given by convolutions, the equivalence (i) $\iff$ (ii) shows that we can recover $m\in \ell^1(\Lambda)$ from $\mathcal{G}_{m}^{\varphi,\xi}$ under the weaker condition (i) -- this holds in particular for finite sequences $m$.

Finally, we apply our techniques to prove a version of Wiener's division lemma in Theorem \ref{thm:underspread}. At the level of Weyl symbols this turns out to reproduce a result by Gr\"ochenig and Pauwels \cite{Grochenig:2014}, but in our context it has the following interpretation:
\begin{quote}
	If $\F_W(S)$ has compact support for some operator $S$, and the support is sufficiently small compared to the density of $\Lambda$, then there exists a sequence $m\in \ell^\infty(\Lambda)$ such that $S=m\star_\Lambda A$ for some $A\in \beauty$. If $S$ belongs to the Schatten-$p$ class of compact operators, then $m\in \ell^p(\Lambda)$. 
\end{quote}
 The above result fits well into the common intuition that operators $S$ with compactly supported $\F_W(S)$ (so-called underspread operators) can be approximated by Gabor multipliers \cite{Dorfler:2010} -- i.e. by operators $c\star_\Lambda T$ where $T$ is a rank-one operator. The result shows that if we allow $T$ to be \textit{any} operator in $\beauty$, then any underspread operator $S$ is precisely of the form $S=c\star_\Lambda T$ for a sufficiently dense lattice $\Lambda$. 
 
 We end this introduction by emphasizing the hybrid nature of our setting. In \cite{Werner:1984}, Werner introduced quantum harmonic analysis of functions on $\Rdd$ and operators on the Hilbert space $L^2(\Rd)$. We are considering the discrete setting of sequences on a lattice instead of functions on $\Rdd$. If we had modified the Hilbert space $L^2(\Rd)$ accordingly, many of our results would follow by the arguments of \cite{Werner:1984}, as already outlined in \cite{Kiukas:2012}. However, we keep the same Hilbert space $L^2(\Rd)$ as in the continuous setting. We are therefore mixing the discrete (lattices) and the continuous ($L^2(\Rd)$), which leads to some extra intricacies. 
 
\section{Conventions}
By a lattice $\Lambda$ we mean a full-rank lattice in $\Rdd$, i.e. $\Lambda=A\Z^{2d}$ for $A\in GL(2d,\R)$. The volume of $\Lambda=A\Z^{2d}$ is $|\Lambda|:=\det(A)$. For a lattice $\Lambda$, the Haar measure on $\Rdd/\Lambda$ will always be normalized so that $\Rdd/\Lambda$ has total measure $1$. 

If $X$ is a Banach space and $X'$ its dual space, the action of $y\in X'$ on $x\in X$ is denoted by the bracket $\inner{y}{x}_{X',X}$, where the bracket is antilinear in the second coordinate to be compatible with the notation for inner products in Hilbert spaces. This means that we are identifying the dual space $X'$ with \textit{anti}linear functionals on $X$. For two Banach spaces $X,Y$ we use $\mathcal{L}(X,Y)$ to denote the Banach space of continuous linear operators from $X$ to $Y$, and if $X=Y$ we simply write $\mathcal{L}(X)$. The notation $P \lesssim Q$ means that there is some $C>0$ such that $P\leq C\cdot Q$. 

\section{Spaces of operators and functions}
\subsection{Time-frequency shifts and the short-time Fourier transform}
For $z= (x,\omega)\in \Rdd$ we define the \textit{time-frequency shift} operator $\pi(z)$ by
\begin{equation*}
  (\pi(z)\psi)(t)=e^{2\pi i \omega \cdot t}\psi(t-x) \quad \text{ for } \psi \in L^2(\Rd).
\end{equation*}
Hence $\pi(z)$ can be written as the composition $M_\omega T_x$ of a translation operator $(T_x\psi)(t)=\psi(t-x)$ and a modulation operator $(M_\omega \psi)(t)=e^{2\pi i \omega \cdot t}\psi(t)$. 
The time-frequency shifts $\pi(z)$ are unitary operators on $L^2(\Rd)$. For $\psi,\varphi\in L^2(\Rd)$ we can use the time-frequency shifts to define the \textit{short-time Fourier transform } $V_\varphi \psi$ of $\psi$ with window $\varphi$ by 
\begin{equation*} 
  V_\varphi \psi (z)=\inner{\psi}{\pi(z)\varphi}_{L^2} \quad \text{ for } z\in \Rdd.
\end{equation*}
The short-time Fourier transform satisfies an orthogonality condition, sometimes called Moyal's identity \cite{Grochenig:2001,Folland:1989}.
\begin{lem}[Moyal's identity]
If $\psi_1, \psi_2, \varphi_1, \varphi_2 \in L^2(\R^d)$, then $V_{\varphi_i}\psi_j \in L^2(\R^{2d})$ for $i,j\in \{1,2\}$, and the relation
\begin{equation*}
	\inner{V_{\varphi_1}\psi_1}{V_{\varphi_2}\psi_2}_{L^2}=\inner{\psi_1}{\psi_2}_{L^2}\overline{\inner{\varphi_1}{\varphi_2}}_{L^2}
\end{equation*}
holds, where the leftmost inner product is in $L^2(\R^{2d})$ and those on the right are in $L^2(\Rd)$.
\end{lem}
By replacing the inner product in the definition of $V_\varphi \psi$ by a duality bracket, one can define the short-time Fourier transform for other classes of $\psi,\varphi$. The most general case we need is that of a Schwartz function $\varphi \in \mathcal{S}(\Rd)$ and a tempered distribution $\psi \in \mathcal{S}'(\Rd)$; we define 
\begin{equation*} 
  V_\psi \varphi (z)=\inner{\psi}{\pi(z)\varphi}_{\mathcal{S}',\mathcal{S}} \quad \text{ for } z\in \Rdd.
\end{equation*}

\subsection{Feichtinger's algebra}
 An appropriate space of functions for our purposes will be Feichtinger's algebra $S_0(\Rd)$, first introduced by Feichtinger in \cite{Feichtinger:1981}. To define $S_0(\Rd)$, let $\varphi_0$ denote the $L^2$-normalized Gaussian $\varphi_0(x)=2^{d/4}e^{-\pi x\cdot x}$ for $x\in \Rd$. Then $S_0(\Rd)$ is the space of all $\psi\in \mathcal{S}'(\Rd)$ such that 
  \begin{equation*}
  \|\psi\|_{S_0}:=\int_{\Rdd} |V_{\varphi_0}\psi(z)| \ dz <\infty.
\end{equation*}
 With the norm above, $S_0(\Rd)$ is a Banach space of continuous functions and an algebra under multiplication and convolution \cite{Feichtinger:1981}. By \cite[Thm. 11.3.6]{Grochenig:2001}, the dual space of $S_0(\Rd)$ is the space $S_0'(\Rd)$ consisting of all $\psi \in \mathcal{S'}(\Rd)$ such that 
 \begin{equation*}
  \|\psi\|_{S_0'}:=\sup_{z\in \Rdd} |V_{\varphi_0}\psi(z)| \ dz <\infty,
\end{equation*}
where an element $\psi\in S_0'(\Rd)$ acts on $\phi\in S_0(\Rd)$ by 
\begin{equation*}
  \inner{\phi}{\psi}_{S_0',S_0}=\int_{\Rdd} V_{\varphi_0}\phi(z) \overline{V_{\varphi_0}\psi(z)} \ dz.
\end{equation*}
We get the following chain of continuous inclusions:
\begin{equation*}
  \mathcal{S}(\Rd)\hookrightarrow S_0(\Rd) \hookrightarrow L^2(\Rd) \hookrightarrow S_0'(\Rd) \hookrightarrow \mathcal{S}'(\Rd).
\end{equation*}

One important reason for using Feichtinger's algebra is that it consists of continuous functions, and that sampling them over a lattice produces a summable sequence \cite[Thm. 7C)]{Feichtinger:1981}.

\begin{lem}[Sampling Feichtinger's algebra] \label{lem:s0sampling} 
	Let $\Lambda$ be a lattice in $\Rdd$ and $f\in S_0(\Rdd)$. Then $f\vert_\Lambda = \{f(\lambda)\}_{\lambda \in \Lambda}\in \ell^1(\Lambda)$ with
	\begin{equation*}
		\|f \vert_{\Lambda}\|_{\ell^1} \lesssim \|f\|_{S_0},
	\end{equation*}
	where the implicit constant depends only on the lattice $\Lambda.$
\end{lem}

\subsection{The symplectic Fourier transform}
We will use the \textit{symplectic Fourier transform}  $\F_\sigma f$ of functions $f\in L^1(\Rdd)$, defined by 
\begin{equation*}
  \F_{\sigma}f(z)=\int_{\R^{2d}} f(z') e^{-2 \pi i\sigma(z,z')} \ dz',
\end{equation*}
where $\sigma$ is the standard symplectic form $\sigma(z,z')=\omega\cdot x'-x\cdot \omega'$ for $z=(x,\omega),z'=(x',\omega').$ $\F_\sigma$ is a Banach space isomorphism $S_0(\Rdd)\to S_0(\Rdd)$, extends to a unitary operator $L^2(\Rdd)\to L^2(\Rdd)$ and a Banach space isomorphism $S_0'(\Rdd)\to S_0'(\Rdd)$ \cite[Lem. 7.6.2]{Feichtinger:1998}. In fact, $\F_\sigma$ is its own inverse, so that $\F_\sigma(\F_{\sigma}(f))=f$ for $f\in S_0'(\Rdd)$ \cite[Prop. 144]{deGosson:2011}.

\subsection{Banach spaces of operators on $L^2(\Rd)$}
The results of this paper concern operators on various function spaces, and we will pick operators from two kinds of spaces: the Schatten-$p$ classes $\SC^p$ for $1\leq p \leq \infty$ and a space $\beauty$ of operators defined using the Feichtinger algebra. 
\subsubsection{The Schatten classes}
Starting with the Schatten classes, we recall that 
 any compact operator $S$ on $L^2(\Rd)$ has a singular value decomposition  \cite[Remark 3.1]{Busch:2016}, i.e. there exist two orthonormal sets $\{\psi_n\}_{n\in \mathbb{N}}$ and $\{\phi_n\}_{n\in \mathbb{N}}$ in $L^2(\Rd)$ and a bounded sequence of positive numbers $\{s_n(S)\}_{n\in \mathbb{N}}$  such that $S$ may be expressed as 
\begin{equation*}
S = \sum\limits_{n \in \mathbb{N}} s_n(S) \psi_n\otimes \phi_n, 
\end{equation*}
with convergence of the sum in the operator norm. Here $\psi
 \otimes \phi$ for $\psi,\phi\in L^2(\Rd)$ denotes the rank-one operator $\psi\otimes \phi (\xi)=\inner{\xi}{\phi}_{L^2} \psi$.

For $1\leq p<\infty$ we define the \textit{Schatten-$p$ class} $\SC^p$ of operators on $L^2(\Rd)$ by $$\SC^p=\lbrace T\text{ compact}: \{s_n(T)\}_{n\in \mathbb{N}} \in \ell^p\rbrace.$$ 
	To simplify the statement of some results, we also define $\SC^{\infty}=\bo$ with $\|\cdot\|_{\SC^\infty}$ given by the operator norm. The Schatten-$p$ class $\SC^p$ is a Banach space with the norm $\|S\|_{\SC^p}=\left(\sum\limits_{n\in \mathbb{N}} s_n(S)^p\right)^{1/p}$.
Of particular interest is the space $\tco:=\tco^1$; the so-called trace class operators. Given an orthonormal basis $\{e_n\}_{n\in \N}$ of $L^2(\Rd)$, the trace defined by $$\tr(S)=\sum_{n\in \N} \inner{Se_n}{e_n}_{L^2}$$ is a well-defined and bounded linear functional on $\tco$, and independent of the orthonormal basis $\{e_n\}_{n\in \N}$ used. The dual space of $\tco$ is $\bo$ \cite[Thm. 3.13]{Busch:2016}, and $T\in \bo$ defines a bounded antilinear functional on $\tco$ by 
\begin{equation*}
  \inner{T}{S}_{\bo,\tco}=\tr(TS^*) \quad \text{ for } S\in \tco.
\end{equation*}

Another special case is the space of Hilbert-Schmidt operators $\HS:=\tco^2$, which is a Hilbert space with inner product $$\inner{S}{T}_{\HS}=\tr(ST^*).$$
\subsubsection{The Weyl transform and operators with symbol in $S_0(\Rdd)$}
The other class of operators we will use will be defined in terms of the \textit{Weyl transform.} We first need the \textit{cross-Wigner distribution} $W(\xi,\eta)$ of two functions $\xi,\eta \in L^2(\Rd)$,  defined by
{\small \begin{equation*}
  W(\xi,\eta)(x,\omega)=\int_{\R^d} \xi\left(x+\frac{t}{2}\right)\overline{\eta\left(x-\frac{t}{2}\right)} e^{-2 \pi i \omega \cdot t} \ dt \quad \text{ for } (x,\omega)\in \Rdd.
 \end{equation*} }
 For $f \in S_0'(\R^{2d})$, we define the \textit{Weyl transform} $L_{f}$ of $f$ to be the operator $L_f:S_0(\Rd)\to S_0'(\Rd)$ given by 
\begin{equation*}
  \inner{L_{f}\eta}{\xi}_{S_0',S_0}:=\inner{f}{W(\xi,\eta)}_{S_0',S_0} \quad \text{ for any } \xi,\eta \in S_0(\R^d).
\end{equation*}
$f$ is called the \textit{Weyl symbol} of the operator $L_{f}$.
 By the kernel theorem for modulation spaces \cite[Thm. 14.4.1]{Grochenig:2001}, the Weyl transform is a bijection from $S_0'(\Rdd)$ to $\mathcal{L}(S_0(\Rd),S_0'(\Rd))$.
 \begin{notation}
 In particular, any $S\in \mathcal{L}(S_0(\Rd),S_0'(\Rd))$ has a Weyl symbol, and we will denote the Weyl symbol of $S$ by $\weyl_S$. By definition, this means that $L_{\weyl_S}=S$.
\end{notation}
It is also well-known that the Weyl transform is a unitary mapping from $L^2(\Rdd)$ to $\HS$\cite{Pool:1966}. This means in particular that \begin{equation*}
  \inner{S}{T}_{\HS} = \inner{\weyl_S}{\weyl_T}_{L^2} \quad \text{ for } S,T \in \HS,
\end{equation*}
which often allows us to reduce statements about Hilbert Schmidt operators to statements about $L^2(\Rdd)$.

We then define $\beauty$ to be the Banach space of continuous operators $S:S_0(\Rd)\to S_0'(\Rd)$  such that $\weyl_S\in S_0(\Rdd)$, with norm $$\|S\|_{\beauty}:=\|\weyl_S\|_{S_0}.$$ $\beauty$ consists of trace class operators $L^2(\Rd)$ and we have a norm-continuous inclusion $\iota:\beauty \hookrightarrow \tco$\cite{Grochenig:1996,Grochenig:1999}. 
\begin{exmp}
	If $\phi,\psi \in L^2(\Rd)$, consider the rank-one operator $\phi\otimes \psi.$ Its Weyl symbol is the cross-Wigner distribution $W(\phi,\psi)$\cite[Cor. 207]{deGosson:2011}, and $W(\phi,\psi)\in S_0(\Rdd)$ if and only if $\phi,\psi \in S_0(\Rd)$\cite[Prop. 365]{deGosson:2011}. The simplest examples of operators in $\beauty$ are therefore $\phi\otimes \psi$ for $\phi,\psi \in S_0(\Rd)$.
\end{exmp}

The dual space $\beast$ can also be identified with a Banach space of operators. By definition, $\tau:\beauty \to S_0(\Rdd)$ given by $\tau(S)= \weyl_S$ is an isometric isomorphism. Hence the Banach space adjoint $\tau^*: S_0'(\Rdd)\to \beast$ is also an isomorphism. Since the Weyl transform is a bijection from $S_0'(\Rdd)$ to $\mathcal{L}(S_0(\Rd),S_0'(\Rd))$, we can identify $\beast$ with operators $S_0(\Rd)\to S_0'(\Rd)$:
\begin{equation*} 
   \beast\xleftrightarrow{\ \ \tau^*\ \ }  S_0'(\Rdd) \xleftrightarrow{\text{Weyl calculus}}  \mathcal{L}(S_0(\Rd), S_0'(\Rd)).
\end{equation*} 

%By writing out the details of these identifications one finds that duality is given by
%\begin{equation*} 
 % \inner{T}{S}_{\beast,\beauty} = \inner{\weyl_T}{\weyl_S}_{S_0',S_0} \quad \text{ for } T\in \beast,T\in \beauty
%\end{equation*}
%when we identify elements $T\in \beast$ with operators $S_0(\Rd)\to S_0'(\Rd).$ 
In this paper we will always consider elements of $\beast$ as operators $S_0(\Rd)\to S_0'(\Rd)$ using these identifications.
 Since $\bo$ is the dual space of $\tco$, the Banach space adjoint $\iota^*:\bo\to \beast$ is a weak*-to-weak*-continuous inclusion of $\bo$ into $\beast$.

\begin{rem}
	For more results on $\beauty$ and $\beast$ we refer to \cite{Feichtinger:1998,Feichtinger:2018}. In particular we mention that we could have defined $\beauty$ using other pseudodifferential calculi, such as the Kohn Nirenberg calculus, and still get the same space $\beauty$ with an equivalent norm. We would also like to point out that the statements of this section may naturally be rephrased using the notion of Gelfand triples, see \cite{Feichtinger:1998}.
\end{rem}

\subsection{Translation of operators}
The idea of translating an operator $S\in \bo$ by $z\in \Rdd$ using conjugation with $\pi(z)$ has been utilized both in physics \cite{Werner:1984} and time-frequency analysis \cite{Feichtinger:1998,Kozek:1992a}. More precisely, we define for $z\in \Rdd$ and $S\in \beast$ the translation of $S$ by $z$ to be the operator $$\alpha_z(S)=\pi(z)S\pi(z)^*.$$
We will also need the operation $S\mapsto \check{S}=PSP$, where $P$ is the parity operator $(P\psi)(t)=\psi(-t)$ for $\psi \in L^2(\Rd)$. The main properties of these operations are listed below, note in particular that part $(i)$ supports the intuition that $\alpha_z$ is a translation of operators. See Lemmas 3.1 and 3.2 in \cite{Luef:2018c} for the proofs.

\begin{lem}\label{lem:translation}
	Let $S\in \beast$.
	\begin{enumerate}[(i)]
		\item If $\weyl_S$ is the Weyl symbol of $S$, then the Weyl symbol of $\alpha_z(S)$ is $T_z (\weyl_S).$
		\item $\alpha_z(\alpha_{z'}(S))=\alpha_{z+z'}(S).$
		\item The operations $\alpha_z$, $^*$ and $\check{\ }$ are isometries on $\beauty, \beast$ and $\SC^p$ for $1\leq p \leq \infty$.
		\item $(S^*)\widecheck{\ }= (\check{S})^*$.
	\end{enumerate}
\end{lem}

By the last part we can unambiguously write $\check{S}^*$. 
\section{Convolutions of sequences and operators} \label{sec:convolutions}

In \cite{Werner:1984}, the convolution of a function $f\in L^1(\Rdd)$ and an operator $S\in \tco$ was defined by the operator-valued integral 
\begin{equation*}
  f\star S = \int_{\Rdd} f(z) \alpha_z(S) \ dz
\end{equation*}
and the convolution of two operators $S,T \in \tco$ was defined to be the \textit{function}
\begin{equation*}
  S\star T(z)=\tr(S \alpha_z(\check{T})) \quad \text{ for } z\in \Rdd.
\end{equation*}
These definitions, along with a Fourier transform defined for operators, have been shown to produce a theory of quantum harmonic analysis with non-trivial consequences for topics such as quantum measurement theory \cite{Kiukas:2012} and time-frequency analysis \cite{Luef:2018c}. The setting where $\Rdd$ is replaced by some lattice $\Lambda \subset \Rdd$ is frequently studied in time-frequency analysis, and our goal is therefore to develop a theory of convolutions and Fourier transforms of operators in that setting. 

For a sequence $c\in \ell^1(\Lambda)$ and $S \in \tco$, we define the operator 

\begin{equation} \label{eq:convseqop}
  c \star_\Lambda S:= S\star_{\Lambda} c:= \sum_{\lambda \in \Lambda} c(\lambda) \alpha_{\lambda}(S),
\end{equation}
and for operators $S\in \beauty$ and $T\in \tco$ we define the sequence
\begin{equation} \label{eq:convopop}
  S\star_\Lambda T(\lambda)= S\star T(\lambda) \quad \text{ for } \lambda \in \Lambda.
\end{equation}
Hence $S\star_\Lambda T$ is the \textit{sequence} obtained by restricting the \textit{function} $S\star T$ to $\Lambda$.
\begin{rem}
	We use the same notation $\star_\Lambda$ for the convolution of an operator and a sequence and for the convolution of two operators. The correct interpretation of $\star_\Lambda$ will always be clear from the context. 
\end{rem}

Since $\alpha_\lambda$ is an isometry on $\tco$ and $\beauty$,  $c\star_\Lambda S$ is well-defined with $\|c\star_\Lambda S\|_{\tco}\leq \|c\|_{\ell^1} \|S\|_{\tco}$ for $S\in \tco$ and similarly $\|c\star_\Lambda S\|_{\beauty}\leq \|c\|_{\ell^1} \|S\|_{\beauty}$ for $S\in \beauty$. The fact that $S\star_\Lambda T$ is a well-defined and summable sequence on $\Lambda$ is less straightforward. 

\begin{prop} \label{prop:convwelldefined}
	If $S\in \beauty$ and $T\in \tco$, then $S\star_\Lambda T\in \ell^1(\Lambda)$ with $\|S\star_\Lambda T\|_{\ell^1}\lesssim \|S\|_{\beauty}\|T\|_{\tco}$. 
\end{prop}
\begin{proof}
	By \cite[Thm. 8.1]{Luef:2018c} we know that $S\star T\in S_0(\Rdd)$ with $\|S\star T\|_{S_0}\lesssim \|S\|_{\beauty} \|T\|_\tco.$ Hence the result follows from Lemma \ref{lem:s0sampling} and $S\star_\Lambda T(\lambda)= S\star T(\lambda)$.
\end{proof}

\subsection{Gabor multipliers and sampled spectrograms}
If we consider rank-one operators, these convolutions reproduce well-known objects from time-frequency analysis.
First consider the rank-one operator $\xi_1 \otimes \xi_2$ for $\xi_1,\xi_2\in L^2(\Rd)$. The operators $c\star_\Lambda (\xi_1 \otimes \xi_2)$ are well-known in time-frequency analysis as \textit{Gabor multipliers} \cite{Feichtinger:2002,Feichtinger:2003,Benedetto:2006,Dorfler:2010}: it is simple to show that 
\begin{equation*}
  \alpha_{\lambda} (\xi_1 \otimes \xi_2)=(\pi(\lambda)\xi_1) \otimes (\pi(\lambda)\xi_2),
\end{equation*}
so if $c\in \ell^1(\Lambda)$ it follows from the definition \eqref{eq:convseqop} that $c\star_\Lambda (\xi_1\otimes \xi_2)$ acts on $\psi \in L^2(\Rd)$ by 
\begin{equation}\label{eq:gabormultiplier}
  c\star_\Lambda (\xi_1\otimes \xi_2)\psi =\sum_{\lambda \in \Lambda} c(\lambda)V_{\xi_2}\psi(\lambda)\pi(\lambda)\xi_1,
\end{equation}
which is the definition of the Gabor multiplier $\mathcal{G}_c^{\xi_2,\xi_1}$ used in time-frequency analysis \cite{Feichtinger:2003}, i.e. $\mathcal{G}_c^{\xi_2,\xi_1}=c\star_\Lambda (\xi_1\otimes \xi_2)$.
\begin{rem}
	In this sense, operators of the form $c\star_\Lambda S$ are a generalization of Gabor multipliers. We mention that this is a different generalization from the \textit{multiple Gabor multipliers} introduced in \cite{Dorfler:2010}. 
\end{rem}
If we pick another rank-one operator $\check{\varphi_1}\otimes \check{\varphi_2}$ for $\varphi_1,\varphi_2\in L^2(\Rd)$ (here $\check{\varphi}(t)=\varphi(-t)$), one can calculate using the definition \eqref{eq:convopop} that
\begin{equation} \label{eq:tworankone0}
  (\xi_1 \otimes \xi_2)\star_\Lambda (\check{\varphi_1}\otimes \check{\varphi_2})(\lambda)=V_{\varphi_2} \xi_1(\lambda)\overline{V_{\varphi_1}\xi_2(\lambda)}.
\end{equation}
In particular, if $\varphi_1=\varphi_2=\varphi$ and $\xi_1=\xi_2=\xi$, then 
\begin{equation} \label{eq:tworankone}
 ( \xi\otimes \xi) \star_\Lambda (\check{\varphi}\otimes \check{\varphi}) (\lambda)=|V_\varphi \xi(\lambda)|^2.
\end{equation}
The function $|V_\varphi \xi (z)|^2$ is the so-called spectrogram of $\xi$ with window $\varphi$, hence $(\xi\otimes \xi) \star_\Lambda (\check{\varphi}\otimes \check{\varphi})$ consists of samples of the spectrogram over $\Lambda$. 

Finally, if $S\in \tco$ is any operator, then one may calculate that 
\begin{equation} \label{eq:generalwithrankone}
  S\star_\Lambda \check{\varphi_1}\otimes \check{\varphi_2} (\lambda)=\inner{S\pi(\lambda)\varphi_1}{\pi(\lambda)\varphi_2}_{L^2},
\end{equation}
often called the lower symbol of $S$ with respect to $\varphi_1,\varphi_2$ and $\Lambda$ \cite{Feichtinger:2002}.
 \begin{rem}
	In particular, Proposition \ref{prop:convwelldefined} does not hold for all $S\in \tco$. By Remark 4.6 in \cite{Benedetto:2006}, there exists a function $\psi \in L^2(\R)$ such that $$\sum_{(m,n)\in \Z^2} (\psi \otimes \psi)\star_{\Z^2} (\check{\psi} \otimes \check{\psi})(m,n)=\sum_{(m,n)\in \Z^2} |V_\psi \psi (m,n)|^2 =\infty.$$ Since $\psi \otimes \psi,\check{\psi}\otimes \check{\psi} \in \tco$, this shows that the assumption $S\in \beauty$ in Proposition \ref{prop:convwelldefined}  is necessary.
\end{rem}  

\subsection{Associativity and commutativity of convolutions}
Since the convolution $S\star T$ of two operators $S,T\in \tco$ is commutative in the continuous setting\cite[Prop. 3.2]{Werner:1984}, it follows from the definitions that the convolutions \eqref{eq:convseqop} and \eqref{eq:convopop} are commutative. It is also a straightforward consequence of the definitions that the convolutions are bilinear. 

In the original theory of Werner \cite{Werner:1984}, the associativity of the convolution operations is of fundamental importance. Associativity still holds in some cases when moving from $\Rdd$ to $\Lambda$, but we will later see in Corollary \ref{cor:noassociativity} that the convolution of three operators over a lattice is not associative in general. In what follows, $c\ast_\Lambda d$ denotes the usual convolution of sequences 
\begin{equation*}
  c\ast_\Lambda d(\lambda)=\sum_{\lambda'\in \Lambda} c(\lambda')d(\lambda-\lambda').
\end{equation*}

\begin{prop}[Associativity]
	Let $c,d\in \ell^1(\Lambda)$, $S \in \beauty$ and $T\in \tco$. Then
	\begin{enumerate}[(i)]
		\item $c \ast_\Lambda (S\star_\Lambda T)=(c\star_\Lambda S)\star_\Lambda T$,
		\item $(c\ast_\Lambda d)\star_\Lambda T=c\star_\Lambda(d\star_\Lambda T)$. 
	\end{enumerate}
\end{prop}
\begin{proof}
For the proof of $(i)$, we write out the definitions of the convolutions and use the commutativity $S\star_\Lambda T=T\star_\Lambda S$ to get
\begin{align*}
   c\ast_\Lambda (S\star_\Lambda T)(\lambda)&=  c\ast_\Lambda (T\star_\Lambda S)(\lambda) \\
   &= \sum_{\lambda'\in \Lambda} c(\lambda') \tr(T\alpha_{\lambda-\lambda'}(\check{S}))  \\
   &= \tr\left(T \sum_{\lambda'\in \Lambda} c(\lambda') \alpha_{\lambda-\lambda'}(\check{S}) \right) \\
   &= \tr\left(T \alpha_{\lambda} \left(\sum_{\lambda'\in \Lambda} c(\lambda') \alpha_{-\lambda'}(\check{S})\right) \right) \quad \text{ by Lemma \ref{lem:translation}}  \\
   &= \tr\left(T \alpha_{\lambda} \left(P \sum_{\lambda'\in \Lambda} c(\lambda') \alpha_{\lambda'}(S) P \right) \right) \\
   &= T\star_\Lambda (c\star_\Lambda S) \quad \text{ by \eqref{eq:convseqop} and \eqref{eq:convopop}}\\
   &=(c\star_\Lambda S) \star_\Lambda T \quad \text{ by commutativity}.
\end{align*}
We have used the easily checked relation $\alpha_{-\lambda'} (\check{S})=P\alpha_{\lambda'} (S) P$. 
%To justify taking the trace outside the sum, recall that since $T\in \beauty\subset \bo$, the expression $A\mapsto \tr(TA)$ defines a bounded, linear functional on $\tco$ \cite{Simon:2010}, and it is straightforward to check that $\sum_{\lambda'\in \Lambda}c(\lambda') \alpha_{\lambda-\lambda'}(\check{S})$ converges absolutely in $\tco$ by the summability of $c$.
For the second part, we find that
\begin{align*}
		(c\ast_\Lambda d)\star_\Lambda T&= \sum_{\lambda\in \Lambda} (c\ast_\Lambda d)(\lambda) \alpha_\lambda(T) \\
		&= \sum_{\lambda\in \Lambda} \sum_{\lambda'\in \Lambda} c(\lambda')
d(\lambda-\lambda')\alpha_\lambda(T) \\
&= \sum_{\lambda'\in \Lambda} c(\lambda') \sum_{\lambda \in \Lambda}
d(\lambda-\lambda')\alpha_\lambda(T) \\
&= \sum_{\lambda'\in \Lambda} c(\lambda') \alpha_{\lambda'} (d\star_\Lambda T)
=c \star_\Lambda (d\star_\Lambda T).
	\end{align*}
	To pass to the last line we have used the relation $\alpha_{\lambda '} (d\star_\Lambda T)=\sum_{\lambda} d(\lambda-\lambda ')\alpha_\lambda (T)$, which is easily verified. 
\end{proof}
\begin{rem}
Part $(ii)$ of this result along with the trivial estimate $\|c\star_\Lambda T\|_{\tco}\leq \|c\|_{\ell^1} \|T\|_\tco$ shows that $\tco$ is a \textit{Banach module} (see \cite{Graven:1974}) over $\ell^1(\Lambda)$ if we define the action of $c\in \ell^{1}(\Lambda)$ on $T\in \tco$ by $c\star_\Lambda T$. The same proofs also show that this is true when $\tco$ is replaced by $\beauty$ or any Schatten class $\SC^p$ for $1\leq p \leq \infty$.	
\end{rem}

\begin{exmp}
Let $\varphi,\xi \in L^2(\Rd)$ and $c\in \ell^1(\Lambda)$, and define $S=\xi\otimes \xi$ and $T=\check{\varphi}\otimes \check{\varphi}$. If we use  \eqref{eq:tworankone} to simplify $S \star_\Lambda T$ and \eqref{eq:generalwithrankone} to simplify $(c\star_\Lambda S)\star_\Lambda T$, the first part of the result above becomes
\begin{equation} \label{eq:cnn}
  c\ast_\Lambda |V_\varphi \xi|^2(\lambda)=\inner{(c\star_\Lambda \xi\otimes \xi)\pi(\lambda)\varphi}{\pi(\lambda)\varphi}_{L^2}.
\end{equation}
In words, the convolution of a sequence $c$ with samples of a spectrogram $|V_\varphi \xi|^2$ can be described using the action of a Gabor multiplier $c\star (\xi\otimes \xi)$. In applications of convolutional neural networks to audio processing, one often considers the spectrogram of an audio signal as the input to the network. Convolutions of sequences with samples of  spectrograms therefore appear naturally in such networks, and the connection \eqref{eq:cnn} has been exploited in this context -- see the proof of \cite[Thm. 1]{Dorfler:2018}.
\end{exmp}
\subsection{Young's inequality}
The convolutions in \eqref{eq:convseqop} and \eqref{eq:convopop} can be defined for more general sequences and operators by establishing a version of Young's inequality \cite[Thm. 1.2.1]{Grochenig:2001}. In the continuous case such an inequality was established by Werner \cite{Werner:1984} using the $L^p$-norms of functions and Schatten-$p$-norms of operators. In the discrete case, it is not always possible to use the Schatten-$p$-norms, since Proposition \ref{prop:convwelldefined} requires $S\in \beauty$. We will therefore always require that one of the operators belongs to $\beauty$.

 A Young's inequality for Schatten classes can then be established by first extending the domains of the convolutions by duality. If $S\in \beauty$ and $c\in \ell^\infty(\Lambda)$, we define $c\star_\Lambda S\in \bo$ by 
\begin{equation} \label{eq:dualconvolutions}
  \inner{c\star_{\Lambda} S}{R}_{\bo,\tco}:=\inner{c}{ R \star_{\Lambda} \check{S}^* }_{\ell^\infty,\ell^1} \text{ for any } R\in \tco.
\end{equation}
and if $S\in \beauty$ and $T\in \bo=\SC^\infty$ we define $T\star_\Lambda S\in \ell^\infty(\Lambda)$ by
\begin{equation} \label{eq:dualconvolutions2}
  \inner{T\star_\Lambda S}{c}_{\ell^\infty(\Lambda),\ell^1(\Lambda)}:=\inner{T}{c\star_{\Lambda} \check{S}^*}_{\bo,\tco} \text{ for any } c\in \ell^1(\Lambda).
\end{equation}
It is a simple exercise to show that these definitions define elements of $\bo$ and $\ell^\infty(\Lambda)$ satisfying $\|c\star_{\Lambda} S\|_{\bo}\lesssim \|c\|_{\ell^\infty}\|S\|_{\beauty}$ and $\|T\star_\Lambda S\|_{\ell^\infty}\leq \|T\|_{\bo}\|S\|_\beauty$, and that they agree with \eqref{eq:convseqop} and \eqref{eq:convopop} when $c\in \ell^1 (\Lambda)$ or $T\in \tco$.
A standard (complex) interpolation argument then gives the following result, since $(\ell^1(\Lambda),\ell^{\infty}(\Lambda))_{\theta}=\ell^p(\Lambda)$ and $(\SC^1,\SC^{\infty})_{\theta}=\SC^{p}$ with $\frac{1}{p}=1-\theta$ \cite{Bergh:1976}. For Gabor multipliers the second part of this result is well-known  \cite[Thm. 5.4.1]{Feichtinger:2003}, and a weaker version of the first part is known for $p=1,2,\infty$ \cite[Thm. 5.8.3]{Feichtinger:2003}.

\begin{prop}[Young's inequality] \label{prop:youngschatten}
	Let $S\in \beauty$ and $1\leq p \leq \infty$.
	\begin{enumerate}[(i)]
		\item If $T\in \SC^p$, then $\|T\star_\Lambda S\|_{\ell^p}\lesssim \|T\|_{\tco^p}\|S\|_{\beauty}$.
		\item If $c\in \ell^p(\Lambda)$, then $\|c\star_\Lambda S\|_{\SC^p}\lesssim \|c\|_{\ell^p}\|S\|_{\beauty}$.
	\end{enumerate}
\end{prop}
\begin{rem}
	If $1\in \ell^\infty(\Lambda)$ is given by $1(\lambda)=1$ for any $\lambda$, then Feichtinger observed in \cite[Thm. 5.15]{Feichtinger:2002} that $\phi\in S_0(\Rd)$ generates a so-called tight Gabor frame if and only if the Gabor multiplier $1\star_\Lambda (\phi\otimes \phi)$ is the identity operator $I$ in $\bo$. A similar result holds in the more general case: if $S\in \beauty$, then $1\star_\Lambda S^*S=I$ if and only if $S$ generates a tight \textit{Gabor g-frame}, recently introduced in \cite{Skrettingland:2019}.
\end{rem}
We may also use duality to define the convolution $T\star_\Lambda S\in \ell^\infty(\Lambda)$ of $S\in \beauty$ with $T\in \beast$ by 
\begin{equation}  \label{eq:dualconvolutions3}
  \inner{T\star_\Lambda S}{c}_{\ell^\infty,\ell^1}:=\inner{T}{c\star_{\Lambda} \check{S}^*}_{\beast,\beauty} \text{ for any } c\in \ell^1(\Lambda),
\end{equation}
which agrees with \eqref{eq:dualconvolutions2} when $T\in \bo \subset \beast$ and satisfies $\|S\star_\Lambda T\|_{\ell^\infty}\leq \|S\|_\beauty \|T\|_{\beast}$.
We end this section by showing that the space $c_0(\Lambda)$ of sequences vanishing at infinity corresponds to compact operators under convolutions with $S\in \beauty$. The second part of this statement is due to Feichtinger \cite[Thm. 5.15]{Feichtinger:2002} for the special case of Gabor multipliers. 

\begin{prop} 
	Let $S\in \beauty.$ If $T$ is a compact operator, then $T\star_\Lambda S\in c_0(\Lambda)$. If $c\in c_0(\Lambda),$ then $c\star_\Lambda S$ is a compact operator on $L^2(\Rd)$. 
\end{prop}
\begin{proof}
	By \cite[Prop. 4.6]{Luef:2018c}, the \textit{function} $T\star S$ belongs to the space $C_0(\Rdd)$ of continuous functions vanishing at infinity. Since $T\star_\Lambda S$ is simply the restriction of $T\star S$ to $\Lambda$, it follows that $T\star_\Lambda S\in c_0(\Lambda)$. For the second part, let $c_N$ be the sequence 
	\begin{equation*}
  c_N(\lambda)=\begin{cases}
  	c(\lambda) \text{ if } |\lambda|<N \\
  	0 \text{ otherwise.}
  \end{cases}
  \end{equation*}
  Then $c_N\star_\Lambda S=\sum_{|\lambda|<N} c(\lambda) \alpha_\lambda(S)$ is a compact operator for each $N\in \N$, and by Proposition \ref{prop:youngschatten} and the bilinearity of convolutions
  \begin{equation*}
  \|c\star_\Lambda S - c_N \star_\Lambda S\|_{\bo}\leq \|c-c_N\|_{\ell^\infty} \|S\|_{\beauty} \to 0 \text{ as } N\to \infty.
\end{equation*}
Hence $c\star_\Lambda S$ is the limit in the operator topology of compact operators, and is therefore itself compact.
\end{proof}

\section{Fourier transforms} \label{sec:fouriertransforms}
In \cite{Werner:1984}, Werner observed that if one defines a Fourier transform of an operator $S\in \tco$ to be the function
\begin{equation*}
  \F_W(S)(z):= e^{-\pi i x\cdot \omega} \tr(\pi(-z)S) \quad \text{ for } z=(x,\omega)\in \Rdd,
\end{equation*}
then the formulas 
\begin{align} \label{eq:FTofconvolutionscontinuous}
  &\F_W(f\star S)=\F_\sigma(f) \F_W(S), && \F_\sigma(S\star T)=\F_W(S)\F_W(T)
\end{align}
hold for $f\in L^1(\Rdd)$ and $S,T\in \tco$.
The transform $\F_W$, called the \textit{Fourier-Wigner transform} (or the Fourier-Weyl transform \cite{Werner:1984}) is an isomorphism $\F_W:\beauty \to S_0(\Rdd)$, can be extended to a unitary map $\F_W:\HS\to L^2(\Rdd)$, and to an isomorphism $\F_W:\beast \to S_0'(\Rdd)$ by defining $\F_W(S)$ for $S\in \beast$ by duality\cite[Cor. 7.6.3]{Feichtinger:1998}:
\begin{equation} \label{eq:fwdual}
  \inner{F_W(S)}{f}_{S_0',S_0}:= \inner{S}{\rho(f)}_{\beast,\beauty}\quad \text{ for any }f\in S_0(\Rdd).
\end{equation}
Here $\rho:S_0(\Rdd)\to \beauty$ is the inverse of $\F_W$. In fact, $\F_W$ and the Weyl transform are related by a symplectic Fourier transform: for any $S\in \beast$ we have 
\begin{equation*}
  \F_W(S)=\F_\sigma (\weyl_S),
\end{equation*}
where $\weyl_S$ is the Weyl symbol of $S$. As an important special case, the Fourier-Wigner transform of a rank-one operator $\phi\otimes \psi$ is
\begin{equation} \label{eq:fwrankone}
  \F_W(\phi\otimes \psi)(x,\omega)=e^{\pi i x\cdot\omega}V_\psi \phi(x,\omega).
\end{equation}

Since we have defined convolutions of operators and sequences, it is natural to ask whether a version of \eqref{eq:FTofconvolutionscontinuous} holds in our setting. We start by defining a suitable Fourier transform of sequences.
\subsubsection*{Symplectic Fourier series}
For the purposes of this paper, we identify the dual group $\widehat{\R^{2d}}$ with $\R^{2d}$ by the bijection $\R^{2d} \ni z\mapsto \chi_z\in \widehat{\R^{2d}}$, where $\chi_z$ is the \textit{symplectic} character\footnote{Phase space, which in this paper is $\Rdd$, is more properly described by (the isomorphic) space $\R^d\times \widehat{\R^d}$. The symplectic characters appear because they are the natural way of identifying the group $\R^d\times \widehat{\R^d}$ with its dual group.} $\chi_{z}(z')=e^{2\pi i \sigma(z,z')}$. Given a lattice $\Lambda \subset \Rdd$, it follows that the dual group of $\Lambda$ is identified with $\Rdd/\Lambda^\circ$ (see \cite[Prop. 3.6.1]{Deitmar:2014}), where $\Lambda^\circ$ is the annihilator group
\begin{align*}
   \Lambda^\circ&=\{\lambda^\circ \in \Rdd : \chi_{\lambda^\circ}(\lambda)=1 \text{ for any } \lambda \in \Lambda\} \\
   &= \{\lambda^\circ \in \Rdd : e^{2\pi i \sigma(\lambda^\circ,\lambda)}=1 \text{ for any } \lambda \in \Lambda\}.
\end{align*}
The group $\Lambda^\circ$ is itself a lattice, namely the so-called \textit{adjoint lattice} of $\Lambda$ from \cite{Feichtinger:1998,Rieffel:1988}.
Given this identification of the dual group of $\Lambda$, the Fourier transform of $c\in \ell^1(\Lambda)$ is the symplectic Fourier series
\begin{equation*}
  \F_\sigma^\Lambda(c)(\dot{z}):=\sum_{\lambda \in \Lambda} c(\lambda)e^{2\pi i \sigma(\lambda,z)}.
\end{equation*}
Here $\dot{z}$ denotes the image of $z\in \Rdd$ under the natural quotient map $\Rdd\to \Rdd/\Lambda^\circ$, so $ \F_\sigma^\Lambda(c)$ is a function on $\Rdd/\Lambda^\circ$. 
If we denote by $A(\Rdd/\Lambda^\circ)$ the Banach space of functions on $\Rdd/\Lambda^\circ$ with symplectic Fourier coefficients in $\ell^1(\Lambda)$, the Feichtinger algebra has the following property \cite[Thm. 7 B)]{Feichtinger:1981}.
 
\begin{lem}  \label{lem:s0periodization}
If $\Lambda$ is a lattice, the \textit{periodization operator} $P_\Lambda:S_0(\Rdd)\to A(\Rdd/\Lambda)$ defined by $$P_\Lambda(f)(\dot{z})=|\Lambda|\sum_{\lambda \in \Lambda} f(z+\lambda)\quad \text{ for } z\in \Rdd$$ is continuous and surjective.
\end{lem}
\begin{rem}
\begin{enumerate}[(i)]
	\item Since $|\Lambda^\circ|=\frac{1}{|\Lambda|}$ \cite[Lem. 7.7.4]{Feichtinger:1998}, we have $$P_{\Lambda^\circ}(f)(\dot{z})=\frac{1}{|\Lambda|}\sum_{\lambda^\circ \in \Lambda} f(z+\lambda^\circ).$$
	\item One may define Feichtinger's algebra $S_0(G)$ for any locally compact abelian group $G$\cite{Feichtinger:1981}. In fact, all our function spaces besides $L^2(\Rd)$ are examples of Feichtinger's algebra, since $S_0(\Lambda)=\ell^1(\Lambda)$ and $S_0(\Rdd/\Lambda^\circ)=A(\Rdd/\Lambda^\circ).$
\end{enumerate}
\end{rem}
When we identify the dual group of $\Lambda$ with $\Rdd/\Lambda^\circ$, the Poisson summation formula for functions in $S_0(\Rdd)$ takes the following form.
\begin{thm}[Poisson summation] \label{thm:poisson}
	Let $\Lambda$ be a lattice in $\Rdd$ and assume that $f\in S_0(\Rdd)$. Then 
	\begin{equation*} 
\frac{1}{|\Lambda|} \sum_{\lambda^\circ \in \Lambda^\circ} f(z+\lambda^\circ) =   \sum_{\lambda \in \Lambda} \F_\sigma(f)(\lambda)e^{2\pi i \sigma(\lambda,z)}\text{ for } z\in  \Rdd.
\end{equation*}
\end{thm}
\begin{proof}
	This is \cite[Thm. 3.6.3]{Deitmar:2014} with $A=\Rdd$, $B=\Lambda^\circ$ and using $(\Lambda^\circ)^\circ=\Lambda.$ To get equality for any $z\in \Rdd$, we use that $\sum_{\lambda^\circ \in \Lambda^\circ} f(z+\lambda^\circ)$ defines a continuous function on $\Rdd/\Lambda^\circ$ by Lemma \ref{lem:s0periodization}. 
\end{proof}

Since $\F_\sigma^\Lambda$ is a Fourier transform it extends to a unitary mapping $\F_\sigma^\Lambda:\ell^2(\Lambda)\to L^2(\Rdd/\Lambda^\circ)$ satisfying
\begin{equation} \label{eq:fourierseriesofconvolution}
  \F_\sigma^\Lambda(c\ast_\Lambda d)=\F_\sigma^\Lambda(c)\F_\sigma^\Lambda(d)
\end{equation}
for $c\in \ell^1(\Lambda)$ and $d\in \ell^2(\Lambda)$.
\subsection{The Fourier transform of $S\star_\Lambda T$}
We now consider a version of \eqref{eq:FTofconvolutionscontinuous} for sequences. The formula for $\F_\sigma^\Lambda(S\star_\Lambda T)$ is a simple consequence of the Poisson summation formula.

\begin{thm} \label{thm:orthogonality} 
	Let $S\in \beauty$ and $T\in \tco$. Then 
	\begin{align*}
		\F_{\sigma}^{\Lambda}(S\star_\Lambda T)(\dot{z})&=\frac{1}{|\Lambda|} \sum_{\lambda^\circ\in \Lambda^\circ}F_W(S)(z+\lambda^\circ)\F_W(T)(z+\lambda^\circ)\\
		&=P_{\Lambda^\circ}(\F_W(S)\F_W(T))(\dot{z})
	\end{align*}
	for any $z\in \Rdd.$
\end{thm}
\begin{proof}
From  \cite[Thm. 8.2]{Luef:2018c}, we know that $S\star T\in S_0(\Rdd)$. Hence $\F_\sigma(S\star T)=\F_W(S)\F_W(T)\in S_0(\Rdd)$ since $\F_\sigma:S_0(\Rdd)\to S_0(\Rdd)$ is an isomorphism. By applying Poisson's summation formula from Theorem \ref{thm:poisson} to $f=\F_W(S)\F_W(T)$, we find that {\footnotesize
\begin{align*}
  \frac{1}{|\Lambda|}  \sum_{\lambda^\circ \in \Lambda^\circ} \F_W(S)(z+\lambda^\circ)\F_W(T)(z+\lambda^\circ) &= \sum_{\lambda \in \Lambda} \F_\sigma (\F_W(S)\F_W(T))(\lambda)e^{2\pi i \sigma(\lambda,z)}\\
  &=  \sum_{\lambda \in \Lambda} S\star_\Lambda T(\lambda)e^{2\pi i \sigma(\lambda,z)},
\end{align*} }
where we used that $\F_\sigma$ is its own inverse to conclude that $$\F_\sigma(\F_W(S)\F_W(T))(\lambda)=\F_\sigma(\F_\sigma(S\star T))(\lambda)=S\star T(\lambda)=S\star_\Lambda T(\lambda).$$  Since $\F_W(S)\F_W(T)\in S_0(\Rdd)$, Theorem \ref{thm:poisson} says that the equation holds for any $z\in \Rdd$.
\end{proof}
\begin{rem}
	Theorem \ref{thm:orthogonality} has also been proved and used in \cite[Cor. A.3]{Lesch:2016} in noncommutative geometry, with stronger assumptions on $S,T$.
\end{rem}
Theorem \ref{thm:orthogonality} has many interesting special cases. We will frequently refer to the following version, which follows since a short calculation using the definition of the Fourier-Wigner transform shows that 
\begin{equation} \label{eq:fwcheckadjoint}
  \F_W(\check{S^*})(z)=\overline{\F_W(S)(z)}.
\end{equation}

\begin{cor} \label{cor:orthogonalabsolute}
	Let $S\in \beauty.$ Then
	\begin{equation*}
  \F_\sigma^\Lambda (S\star_\Lambda \check{S^*})(\dot{z})=\frac{1}{|\Lambda|}\sum_{\lambda^\circ \in \Lambda^\circ} |\F_W(S)(z+\lambda^\circ)|^2 \quad \text{ for any } z\in \Rdd.
\end{equation*}
\end{cor}

\begin{cor} \label{cor:sumofconvolution}
Let $S\in \beauty$ and $T\in \tco$. Then
\begin{equation*}
	\sum_{\lambda\in \Lambda} S\star_\Lambda T(\lambda)=\frac{1}{|\Lambda|} \sum_{\lambda^\circ \in \Lambda^\circ} F_W(S)(\lambda^\circ)F_W(T)(\lambda^\circ).
\end{equation*}
\end{cor}
\begin{proof}
	This follows from Theorem \ref{thm:orthogonality} with $z=0$.
\end{proof}

Now assume that $S$ and $T$ are rank-one operators: $S=\xi_1 \otimes \xi_2$ for $\xi_1, \xi_2 \in S_0(\Rd)$ and $T=\check{\varphi_1}\otimes \check{\varphi_2}$ for $\varphi_1, \varphi_2 \in L^2(\Rd)$. By \eqref{eq:tworankone0} $$S\star_\Lambda T(\lambda)=V_{\varphi_2} \xi_1(\lambda)\overline{V_{\varphi_1}\xi_2(\lambda)},$$
and noting that $T=\check{T_0}^*$ for $T_0=\varphi_2 \otimes \varphi_1$, we can use \eqref{eq:fwrankone} and \eqref{eq:fwcheckadjoint} to find 
\begin{align*}
  \F_W(S)(z)&=e^{\pi i x\cdot\omega}V_{\xi_2} \xi_1(z) \\
  \F_W(T)(z)&=e^{-\pi i x\cdot\omega}\overline{V_{\varphi_1}\varphi_2(z)}
\end{align*}
 Hence Theorem \ref{thm:orthogonality} says that 
\begin{equation*}
  \F_{\sigma}^{\Lambda}(V_{\varphi_2} \xi_1\overline{V_{\varphi_1}\xi_2}\vert_\Lambda )(\dot{z})=\frac{1}{|\Lambda|} \sum_{\lambda^\circ\in \Lambda^\circ}V_{\xi_2} \xi_1(z+\lambda^\circ)\overline{V_{\varphi_1}\varphi_2(z+\lambda^\circ)}.
\end{equation*}
Furthermore, Corollary \ref{cor:sumofconvolution} gives
\begin{equation*}
  \sum_{\lambda\in \Lambda}V_{\varphi_2} \xi_1(\lambda)\overline{V_{\varphi_1}\xi_2(\lambda)}=\frac{1}{|\Lambda|} \sum_{\lambda^\circ\in \Lambda^\circ}V_{\xi_2} \xi_1(\lambda^\circ)\overline{V_{\varphi_1}\varphi_2(\lambda^\circ)},
\end{equation*}
which is the \textit{fundamental identity of Gabor analysis} \cite{Feichtinger:2006,Tolimieri:1994,Janssen:1995,Rieffel:1988}.

\subsection{The Fourier transform of $c\star_\Lambda S$}
When $c\in \ell^1(\Lambda)$, we obtain the expected formula for $\F_W(c\star_\Lambda S)$.

\begin{prop} \label{prop:spreadinggm} 
If $c\in \ell^1(\Lambda)$ and $S\in \tco$, then
\begin{equation*}
  \F_W(c\star_\Lambda S)(z)=\F_\sigma^\Lambda (c)(\dot{z})\F_W(S)(z) \quad \text{ for } z\in \Rdd.
\end{equation*}
\end{prop}
\begin{proof}
	One easily verifies the formula $$\F_W(\alpha_\lambda (S))(z)=e^{2\pi i \sigma(\lambda,z)}\F_W(S)(z),$$ showing that the Fourier transform of a translation is a modulation. Hence
\begin{align*}
	\F_W(c\star_\Lambda S)(z)&= \sum_{\lambda\in \Lambda} c(\lambda) \F_W(\alpha_\lambda (S)) \\
	&= \sum_{\lambda\in \Lambda} c(\lambda) e^{2\pi i \sigma(\lambda,z)}\F_W(S)(z) \\ 
	&= \F_W(S)(z) \sum_{\lambda\in \Lambda} c(\lambda) e^{2\pi i \sigma(\lambda,z)}.
\end{align*}	
To move $\F_W$ inside the sum, we use that the sum $\sum_{\lambda \in \Lambda} c(\lambda)\alpha_\lambda(S)$ converges absolutely in $\tco$, and $\F_W$ is continuous from $\tco$ to $L^\infty(\Rdd)$ by the Riemann-Lebesgue lemma for $\F_W$ \cite[Prop. 6.6]{Luef:2018c}. 
\end{proof}

\subsubsection{Technical intermezzo} Let $A'(\Rdd/\Lambda^\circ)$ denote the dual space of $A(\Rdd/\Lambda^\circ)$, consisting of distributions on $\Rdd/\Lambda^\circ$ with symplectic Fourier coefficients in $\ell^\infty(\Lambda).$ To understand the statement in Proposition \ref{prop:spreadinggm} when $c\in \ell^\infty(\Lambda)$, we need to `extend' distributions in $A'(\Rdd/\Lambda^\circ)$ to distributions in $S_0'(\Rdd)$. When $f\in A(\Rdd/\Lambda^\circ)$ this is  achieved by
 \begin{equation*}
  A(\Rdd/\Lambda^\circ)\ni f\mapsto f\circ q \in S_0'(\Rdd),
\end{equation*}
 where $q:\Rdd\to \Rdd/\Lambda^\circ$ is the natural quotient map. To extend this map to distributions $f\in A'(\Rdd/\Lambda^\circ)$, one can use Weil's formula \cite[(6.2.11)]{Grochenig:1998} to show that for $f\in A(\Rdd/\Lambda^\circ)$ and $g\in S_0(\Rdd)$ one has 
\begin{equation*}
  \inner{f\circ q}{g}_{S_0',S_0}=\inner{f}{P_{\Lambda^\circ}g}_{A'(\Rdd/\Lambda^\circ),A(\Rdd,\Lambda^\circ)}.
\end{equation*}
This shows that the map $f\mapsto f\circ q$ agrees with the Banach space adjoint $P_{\Lambda^\circ}^*:A'(\Rdd/\Lambda^\circ)\to S_0'(\Rdd)$ for $f\in A(\Rdd/\Lambda^\circ)$. The natural way to extend $f\in A'(\Rdd/\Lambda^\circ)$ is therefore to consider $P_{\Lambda^\circ}^*f\in S_0'(\Rdd)$, and by an abuse of notation we will use $f$ to also denote the extension $P_{\Lambda^\circ}^*f$ -- by definition this means that when $f\in A'(\Rdd/\Lambda^\circ)$ is considered an element of $S_0'(\Rdd)$, it satisfies for $g\in S_0(\Rdd)$
\begin{equation} \label{eq:periodicextension}
  \inner{f}{g}_{S_0',S_0}=\inner{f}{P_{\Lambda^\circ}g}_{A'(\Rdd/\Lambda^\circ),A(\Rdd,\Lambda^\circ)}.
\end{equation}
We also remind the reader that for $c\in \ell^\infty(\Lambda)$ one defines $\F_\sigma^\Lambda(c)$ as an element of $A'(\Rdd/\Lambda^\circ)$ by 
\begin{equation} \label{eq:fsigmadual}
  \inner{\F_\sigma^\Lambda(c)}{g}_{A'(\Rdd/\Lambda^\circ),A(\Rdd/\Lambda^\circ)}:=\inner{c}{(\F_\sigma^\Lambda)^{-1}(g)}_{\ell^\infty(\Lambda),\ell^1(\Lambda)},
\end{equation}
where $(\F_\sigma^\Lambda)^{-1}(g)$ are the symplectic Fourier coefficients of $g$. This is \cite[Example 6.8]{Jakobsen:2018} for the group $G=\Rdd/\Lambda^\circ$. Finally, recall that we can multiply $f\in S_0'(\Rdd)$ with $g\in S_0(\Rdd)$ to obtain an element $fg\in S_0'(\Rdd)$ given by 
\begin{equation} \label{eq:productdual}
  \inner{fg}{h}_{S_0',S_0}:= \inner{f}{\overline{g}h}_{S_0',S_0} \quad \text{ for } h\in S_0(\Rdd).
\end{equation}

\subsubsection{The case $c\in \ell^\infty(\Lambda)$}
The technical intermezzo allows us to make sense of the following generalization of Proposition \ref{prop:spreadinggm}. Recall in particular that $\F_\sigma^\Lambda (c)$ is shorthand for the distribution $P_{\Lambda^\circ}^* (\F_\sigma^\Lambda (c))\in S_0'(\Rdd)$.
\begin{prop} \label{prop:spreadinggm2}
	If $c\in \ell^\infty(\Lambda)$ and $S\in \beauty$, then
\begin{equation*}
  \F_W(c\star_\Lambda S)=\F_\sigma^\Lambda (c)\F_W(S) \quad \text{ in } S_0'(\Rdd).
\end{equation*}
\end{prop}
\begin{proof}
	For $h\in S_0(\Rdd)$, we get from \eqref{eq:fwdual}, \eqref{eq:dualconvolutions} and \eqref{eq:fsigmadual} (in that order)
	\begin{align*}
  \inner{\F_W(c\star_\Lambda S)}{h}_{S_0',S_0}&=\inner{c\star_\Lambda S}{\rho(h)}_{\beast,\beauty}  \\
  &= \inner{c}{\rho(h) \star_\Lambda \check{S}^*}_{\ell^\infty(\Lambda),\ell^1(\Lambda)} \\
  &=\inner{\F_\sigma^\Lambda(c)}{\F_\sigma^\Lambda(\rho(h) \star_\Lambda \check{S}^* )}_{A'(\Rdd/\Lambda^\circ),A(\Rdd/\Lambda^\circ)}.
\end{align*}
 By Theorem \ref{thm:orthogonality} we find using \eqref{eq:fwcheckadjoint} that 
\begin{equation*}
  \F_\sigma^\Lambda(\rho(h) \star_\Lambda \check{S}^*)= P_{\Lambda^\circ} (\overline{\F_W(S)} h),
\end{equation*}
where we also used that  $\rho$ is the inverse of $\F_W$.
On the other hand we find using \eqref{eq:productdual} and \eqref{eq:periodicextension} that 

\begin{align*}
  \inner{\F_\sigma^\Lambda (c)\F_W(S)}{h}_{S_0',S_0}&= \inner{\F_\sigma^\Lambda(c)}{\overline{\F_W(S)}h}_{S_0',S_0} \\
  &= \inner{\F_\sigma^\Lambda(c)}{P_{\Lambda^\circ} (\overline{\F_W(S)}h)}_{A'(\Rdd/\Lambda^\circ),A(\Rdd/\Lambda^\circ)} 
\end{align*}
Hence $\inner{\F_\sigma^\Lambda (c)\F_W(S)}{h}_{S_0',S_0}=\inner{\F_W(c\star_\Lambda S)}{h}_{S_0',S_0}$, which implies the statement. 
\end{proof}

\begin{rem}
	For Gabor multipliers $c\star_\Lambda (\psi\otimes \psi)$, Propositions \ref{prop:spreadinggm} and \ref{prop:spreadinggm2} were proved in \cite[Lem. 14]{Dorfler:2010}, and have been used in the theory of convolutional neural networks \cite{Dorfler:2018}.
\end{rem}

\section{Riesz sequences of translated operators in $\HS$} \label{sec:riesz}
Two of the useful properties of the Weyl transform $f\mapsto L_f$ are that it is a unitary transformation from $L^2(\Rdd)$ to the Hilbert-Schmidt operators $\HS$, and that it respects translations in the sense that  
\begin{equation*}
  L_{T_z f}=\alpha_z(L_f) \quad  \text{ for } f\in L^2(\Rdd), z\in \Rdd.
\end{equation*}
As a consequence, statements concerning translates of functions in $L^2(\Rdd)$ can be lifted to statements about translates of operators and convolutions $\star_\Lambda$ in $\HS$. This approach was first used for Gabor multipliers in \cite{Feichtinger:2002,Feichtinger:2003}, and has later been explored in other works \cite{Benedetto:2006,Dorfler:2010} -- we include these results for completeness, and because the proofs and results find natural formulations and generalizations in the framework of this paper. 

For fixed $S\in \HS$ and lattice $\Lambda$, we will be interested in whether $\{\alpha_\lambda(S)\}_{\lambda \in \Lambda}$ is a \textit{Riesz sequence in $\HS$}, i.e. whether there exist $A,B>0$ such that for all finite sequences $c\in \ell^2(\Lambda)$ 
 \begin{equation} \label{eq:rieszoperator}
  A \|c\|^2_{\ell^2(\Lambda)}\leq \left\|\sum_{\lambda \in \Lambda} c(\lambda) \alpha_\lambda(S) \right\|_\HS^2 \leq B \|c\|_{\ell^2(\Lambda)}^2.
\end{equation}

Since the Weyl transform is unitary and preserves translations, if we let $\weyl_S$ be the Weyl symbol of $S$, then \eqref{eq:rieszoperator} is clearly equivalent to the fact that $\{T_\lambda(\weyl_S)\}_{\lambda \in \Lambda}$ is  a Riesz sequence in $L^2(\Rdd)$, meaning that
 \begin{equation*}
  A \|c\|_{\ell^2(\Lambda)}^2\leq \left\|\sum_{\lambda \in \Lambda} c(\lambda) T_\lambda(\weyl_S) \right\|^2_{L^2(\Rdd)} \leq B \|c\|_{\ell^2(\Lambda)}^2,
  \end{equation*}
for finite $c\in \ell^2(\Lambda)$. Following \cite{Feichtinger:2002,Feichtinger:2003,Benedetto:2006,Dorfler:2010} we can use a result from \cite{Benedetto:1998} to give a characterization of when \eqref{eq:rieszoperator} holds in terms of an expression familiar from Corollary \ref{cor:orthogonalabsolute}.
\begin{thm} \label{thm:rieszsequence}
	Let $\Lambda$ be a lattice and $S\in \beauty$. Then the following are equivalent.
	\begin{enumerate}[(i)]
		\item The function $$\F_\sigma^\Lambda (S\star_\Lambda \check{S}^*) =P_{\Lambda^\circ}(|\F_W(S)|^2)$$ has no zeros in $\Rdd/\Lambda^\circ$.
		\item $\{\alpha_\lambda (S)\}_{\lambda \in \Lambda}$ is a Riesz sequence in $\HS$.
	\end{enumerate}
\end{thm}

\begin{proof}
The equality in $(i)$ is Corollary \ref{cor:orthogonalabsolute}. By the preceding discussion, $\{\alpha_\lambda(S)\}_{\lambda \in \Lambda}$ is a Riesz sequence in $\HS$ if and only if  $\{T_\lambda(\weyl_S)\}_{\lambda \in \Lambda}$ is a Riesz sequence in $L^2(\Rdd)$. 
 The result from \cite{Benedetto:1998} (see \cite{Benedetto:2006} for a statement for general lattices and symplectic Fourier transform) says that $\{T_\lambda(\weyl_S)\}_{\lambda \in \Lambda}$ is a Riesz sequence if and only if there exist $A,B>0$ such that 
\begin{equation*}
  A\leq \frac{1}{|\Lambda|} \sum_{\lambda^\circ\in \Lambda^\circ} |\F_\sigma(\weyl_S)(z+\lambda^\circ)|^2 \leq B \text{ for any } z\in \Rdd.
\end{equation*}
Since the Weyl transform and Fourier-Wigner transform are related by $\F_\sigma(\weyl_S)=\F_W(S)$, we we may restate this condition as 
\begin{equation} \label{eq:rieszinproof}
  A\leq \frac{1}{|\Lambda|}  \sum_{\lambda^\circ\in \Lambda^\circ} |\F_W(S)(z+\lambda^\circ)|^2 \leq B \text{ for any } z\in \Rdd.
\end{equation}
Note that the middle term is $ P_{\Lambda^\circ}(|\F_W(S)|^2)(\dot{z})$, and since $S\in \beauty$ we know that $|\F_W(S)|^2\in S_0(\Rdd)$. Therefore $P_{\Lambda^\circ}(|\F_W(S)|^2)\in A(\Rdd/\Lambda^\circ)$ by Lemma \ref{lem:s0periodization}, which in particular means that $P_{\Lambda^\circ}(|\F_W(S)|^2)$ is a continuous function on the compact space $\Rdd/\Lambda^\circ$. For a continuous function on a compact space, condition \eqref{eq:rieszinproof} is equivalent to having no zeros. This completes the proof. 
\end{proof}

\begin{rem}
\begin{enumerate}[(i)]
	\item Since we assume $S\in \beauty$, the first condition above is in fact equivalent to $\left\{ \alpha_\lambda(S) \right\}_{\lambda \in \Lambda}$ generating a \textit{frame sequence} in $\HS$, which is a weaker statement than (2) above. The proof of this in \cite{Benedetto:2006} for Gabor multipliers works in our more general setting. 
	\item As mentioned in the introduction, Feichtinger \cite{Feichtinger:2002} used the Kohn-Nirenberg symbol rather than the Weyl symbol. This makes no difference for our purposes -- we have opted for the Weyl symbol as it is related to $\F_W$ by a symplectic Fourier transform. 
\end{enumerate}
	\end{rem}
If $\{\alpha_\lambda(S)\}_{\lambda \in \Lambda}$ is a Riesz sequence in $\HS$, the \textit{synthesis operator} is the map $D_S:\ell^2(\Lambda)\to \HS$ given by
\begin{equation*}
  D_S(c)=c\star_\Lambda S=\sum_{\lambda\in \Lambda} c(\lambda) \alpha_\lambda(S),
\end{equation*}
and the sum $\sum_{\lambda\in \Lambda} c(\lambda) \alpha_\lambda(S)$ converges unconditionally in $\HS$ for each $c\in \ell^2(\Lambda)$ \cite[Cor. 3.2.5]{Christensen:2016}. We also get by \cite[Thm. 5.5.1]{Christensen:2016} that 
\begin{equation} \label{eq:closedspanconvolution}
  \overline{\text{span}\{\alpha_\lambda(S):\lambda \in \Lambda\}}=\ell^2(\Lambda)\star S,
\end{equation}
where the closure is taken with respect to the norm in $\HS$.

\subsection{The biorthogonal system and best approximation} \label{sec:biorthogonal}
Any Riesz sequence has a so-called biorthogonal sequence and, by the theory of frames of translates \cite[Prop. 9.4.2]{Christensen:2016}, if the Riesz sequence is of the form $\{\alpha_\lambda(S)\}_{\lambda \in \Lambda}$ for some $S\in \beauty$, then the biorthogonal system has the same form. This means that there exists $S^\prime \in \HS$ such that the biorthogonal system is 
  \begin{equation*}
  \{\alpha_\lambda(S^\prime)\}_{\lambda \in \Lambda},
\end{equation*}
and biorthogonality means that 
\begin{equation*}
  \inner{\alpha_\lambda(S)}{\alpha_{\lambda'}(S')}_{\HS}=\delta_{\lambda,\lambda'},
\end{equation*}
where $\delta_{\lambda,\lambda'}$ is the Kronecker delta.
Now note that for $T\in \HS$ the definition \eqref{eq:convopop} of $T\star_\Lambda S'$ implies that 
\begin{equation*} 
  \inner{T}{\alpha_\lambda(S')}_\HS=T\star_\Lambda \check{S'}^*(\lambda),
\end{equation*}
so if we define $R:=\check{S^\prime}^*$ we have 
\begin{equation} \label{eq:convolutionasinnerproduct} 
  \inner{T}{\alpha_\lambda(S')}_\HS=T\star_\Lambda R(\lambda).
\end{equation}
With this observation we can formulate the standard properties of the biorthogonal sequence using convolutions with $R$. 

\begin{lem} \label{lem:biorthogonal}
	Assume that $\{\alpha_\lambda(S)\}_{\lambda\in \Lambda}$ with $S\in \beauty$ is a Riesz sequence in $\HS$. Let $$V^2:=\overline{\text{span}\{\alpha_\lambda(S):\lambda \in \Lambda\}}=\ell^2(\Lambda)\star S.$$  With $R$ defined as above, we have that 
	\begin{enumerate}[(i)]
		\item $S\star_\Lambda R(\lambda)=\delta_{\lambda,0}.$
\item For any $T\in V^2$, $T\star_\Lambda R\in \ell^2(\Lambda)$.
\item For any $ T\in V^2$, \begin{equation*}
  T=(T\star_\Lambda R)\star_\Lambda S.
\end{equation*}
\end{enumerate}
\end{lem}
\begin{proof}
	This is simply a restatement of the properties of the biorthogonal sequence of a Riesz sequence using the relation $\inner{T}{\alpha_\lambda(S')}_{\HS}=T\star_\Lambda R(\lambda)$ -- with this observation, parts $(i),(ii)$ and $(iii)$ follow from \cite[Thm. 3.6.2]{Christensen:2016}.
\end{proof}
\begin{rem}
\begin{enumerate}[(i)]
	\item If the convolution of three operators were associative, we could find for any $T\in \HS$ (not just $T\in V^2$ as above) that $T=(T\star_\Lambda R)\star_\Lambda S$, since $T\star_\Lambda (R\star_\Lambda S)=T\star_\Lambda \delta_{\lambda,0}=T$. However, we will soon see that the convolution of three operators is \textit{not} associative.
	\item For $T,R\in \HS$, we have strictly speaking not defined $T\star_\Lambda R$ (since \eqref{eq:convopop} has stronger assumptions than simply $\HS$). However, it is clear by the Cauchy Schwarz inequality for $\HS$ that $$|T\star_\Lambda R(\lambda)|=|\inner{T}{\alpha_\lambda(S')}_{\HS}|\leq \|T\|_{\HS} \|S'\|_\HS,$$ so we can define $T\star_\Lambda R\in \ell^\infty(\Lambda)$ by \eqref{eq:convopop} also in this case. 
\end{enumerate}
	
\end{rem}

We will now answer two natural questions. First, to what extent does $R$ inherit the nice properties of $S$ -- is it true that $R\in \beauty$? Then, how is $R$ related to $S$? The answer is provided by the following theorem, first proved by Feichtinger \cite[Thm. 5.17]{Feichtinger:2002} for Gabor multipliers, and the proof finds a natural formulation using our tools.

\begin{thm} \label{thm:biorthogonal}
	Assume that $S\in \beauty$ and that $\left\{ \alpha_{\lambda}(S) \right\}_{\lambda\in \Lambda}$ is a Riesz sequence in $\HS$. If $R$ is defined as above, then $R\in \beauty$ and $R=b \star_\Lambda \check{S}^*$ where $b\in \ell^1(\Lambda)$ are the symplectic Fourier coefficients of 
	\begin{equation*}
  \frac{1}{\F_\sigma^\Lambda
	(S\star_\Lambda \check{S}^*)}=\frac{1}{P_{\Lambda^\circ} \left(|\F_W(S)|^2\right)}.
\end{equation*} 

\end{thm}

\begin{proof}
By \cite[Thm. 3.6.2]{Christensen:2016}, the generator $S^\prime$ of the biorthogonal system belongs to $V^2$, hence there exists some $b^\prime\in \ell^2(\Lambda)$ such that $S^\prime=b^\prime \star_\Lambda S$. Since $R=\check{S^\prime}^*$, one easily checks by the definitions of $\check{\ }$ and $^*$ that
\begin{equation*}
  R=(b'\star S)\check{\ }^*=b\star_\Lambda \check{S}^*
\end{equation*}
if we define $b(\lambda)=\overline{b^\prime(-\lambda)}$. By part $(i)$ of Lemma \ref{lem:biorthogonal} and the associativity of convolutions, we have
\begin{equation*}
  b\ast_\Lambda (\check{S}^* \star_\Lambda S)=(b\star_\Lambda \check{S}^*) \star_\Lambda S=R\star_\Lambda S = \delta_{\lambda,0}.
\end{equation*}
 Taking the symplectic Fourier series of this equation using \eqref{eq:fourierseriesofconvolution} and Corollary \ref{cor:orthogonalabsolute}, we find for a.e. $\dot{z}\in \Rdd/\Lambda^\circ$ 
 \begin{equation*}
  \F_\sigma^\Lambda(b)(\dot{z})\F_\sigma^\Lambda(\check{S}^* \star_\Lambda S)(\dot{z})= \F_\sigma^\Lambda(b)(\dot{z})P_{\Lambda^\circ}\left( |\F_W(S)|^2 \right)(\dot{z})=1,
\end{equation*}
hence 
\begin{equation*}
  \F_\sigma^\Lambda(b)(\dot{z})=\frac{1}{P_{\Lambda^\circ} \left(|\F_W(S)|^2\right)},
\end{equation*}
and by assumption on $S$ (see Theorem \ref{thm:rieszsequence} and its proof) the denominator is bounded from below by a positive constant. Since $S\in \beauty$, we know that $|\F_W(S)|^2\in S_0(\Rdd)$, and therefore Lemma \ref{lem:s0periodization} implies that $P_{\Lambda^\circ} \left(|\F_W(S)|^2\right)\in A(\Rdd/\Lambda^\circ)$. By Wiener's lemma \cite[Thm. 6.1.1]{Reiter:2000}, we get $\frac{1}{P_{\Lambda^\circ} \left(|\F_W(S)|^2\right)}\in A(\Rdd/\Lambda^\circ)$. In other words, $b\in \ell^1(\Lambda)$. Since $b\in \ell^1(\Lambda)$ and $\check{S}^*\in \beauty$, it follows that $R=b\star_\Lambda \check{S}^*\in \beauty$.
\end{proof}

To prepare for the next result, fix $S\in \beauty$ and let $$V^\infty=\ell^\infty(\Lambda)\star_\Lambda S,$$ hence $V^\infty$ is the set of operators given as a convolution $c\star_\Lambda S$ for $c\in \ell^\infty(\Lambda)$. The first part of the next result says that when $\{\alpha_\lambda (S)\}_{\lambda \in \Lambda}$ is a Riesz sequence, then the Schatten-$p$ class properties of $c\star_\Lambda S$ are precisely captured by the $\ell^p$ properties of $c$. This result appears to be a new result even for Gabor multipliers. 
We also determine for any $T\in \HS$ the best approximation (in the norm $\|\cdot\|_{\HS}$) of $T$ by an operator of the form $c\star_\Lambda S$. See \cite[Thm. 5.17]{Feichtinger:2002} and \cite[Thm. 19]{Dorfler:2010} for the statement for Gabor multipliers.

 \begin{cor} \label{cor:banachisomorphism}
	Assume that $S\in \beauty$ and that $\{\alpha_\lambda (S)\}_{\lambda \in \Lambda}$ is a Riesz sequence in $\HS$, and let $R$ be as above.  
	\begin{enumerate}[(i)]
		\item For any $1\leq p \leq \infty$ the map $D_S:\ell^p(\Lambda)\to \SC^p\cap V^\infty$ given by $$D_S(c)= c\star_\Lambda S$$ is a Banach space isomorphism, with inverse $C_R:\SC^p\cap V^\infty\to \ell^p(\Lambda)$ given by $$C_R(T)= T\star_\Lambda R.$$ Hence $V^\infty\cap \SC^p=\ell^{p}(\Lambda)\star_\Lambda S$ and $\|c\|_{\ell^p}\lesssim\|c\star_\Lambda S\|_{\SC^p}\lesssim \|c\|_{\ell^p}$.
		\item For any $T\in \HS$, the best approximation in $\|\cdot\|_{\HS}$ of $T$ by an operator $c\star_\Lambda S$ with $c\in \ell^2(\Lambda)$ is given by  $$c=T\star_\Lambda R.$$ 
		Equivalently, the symplectic Fourier series of $c$ is given by 
 $$\F_\sigma^\Lambda (c)=\frac{P_{\Lambda^\circ}\left[ \overline{\F_W(S)}F_W(T)\right]}{P_{\Lambda^\circ} |\F_W(S)|^2 }.$$ 
	\end{enumerate} 
\end{cor}

 \begin{proof}
\begin{enumerate}[(i)]
	\item 	By Proposition \ref{prop:youngschatten} part $(i)$ we get $\|C_R(T)\|_{\ell^p}\leq \|T\|_{\SC^p} \|R\|_{\beauty}$, and by part $(ii)$ of the same proposition we get $\|D_S(c)\|_{\SC^p}\lesssim \|c\|_{\ell^p}\|S\|_{\beauty}$. Hence both maps in the statement are continuous. It remains to show that the two maps are inverses of each other, which will follow from the associativity of convolutions. First assume that $c\in \ell^p(\Lambda).$ Then
 \begin{equation*}
  C_RD_S(c)=(c\star_\Lambda S)\star_\Lambda R=c\ast_\Lambda (S\star_\Lambda R)=c,
\end{equation*}
	where we have used associativity and part $(i)$ of Lemma \ref{lem:biorthogonal}. Then assume $T\in V^\infty\cap \SC^p$, so that $T=c\star_\Lambda S$ for $c\in \ell^\infty(\Lambda)$. We find 
	\begin{equation*}
  D_SC_R(c\star_\Lambda S)=((c\star_\Lambda S )\star_\Lambda R)\star_\Lambda S= (c\ast_\Lambda (S \star_\Lambda R))\star_\Lambda S=c\star_\Lambda S.
\end{equation*}
Hence $D_S$ and $C_R$ are inverses. In particular $V^\infty \cap \SC^p=\ell^p(\Lambda)\star_\Lambda S$ as $D_S$ is onto $V^\infty \cap \SC^p$, and $V^\infty \cap \SC^p$ is closed in $\SC^p$ (hence a Banach space) since $D_S:\ell^p(\Lambda)\to \SC^p$ has a left inverse $C_R$ and therefore has a closed range in $\SC^p$.
	\item We claim that the map $T\mapsto (T\star_\Lambda R)\star_\Lambda S$ is the orthogonal projection from $\HS$ onto $\ell^2(\Lambda)\star_\Lambda S$, which is a closed subset of $\HS=\SC^2$ by part $(i)$ (or \eqref{eq:closedspanconvolution}). If $T=c\star_\Lambda S$ for some $c\in \ell^2(\Lambda)$, then $c=T\star_\Lambda R$ by part $(i)$ -- therefore $T=(T\star_\Lambda R)\star_\Lambda S$. Then assume that $T\in (\ell^2(\Lambda)\star_\Lambda S)^\perp$. As we saw in \eqref{eq:convolutionasinnerproduct}, we can write 
		\begin{equation} \label{eq:proof:bestapprox}
  T\star_\Lambda R(\lambda)=\inner{T}{\alpha_{\lambda}(S')}_{\HS}.
\end{equation}
From the proof of Theorem \ref{thm:biorthogonal}, $S'=b'\star_\Lambda S$ for some $b'\in \ell^2(\Lambda)$. One easily checks that $$\alpha_\lambda(S')=\alpha_\lambda(b'\star_\Lambda S)=T_{\lambda}b'\star_\Lambda S,$$ where $T_\lambda b'(\lambda')=b'(\lambda'-\lambda)$. It follows that $\alpha_{\lambda} (S')\in \ell^2(\Lambda)\star_\Lambda S$ for any $\lambda \in \Lambda.$ Hence if $T\in (\ell^2(\Lambda)\star_\Lambda S)^\perp,$ \eqref{eq:proof:bestapprox} shows that $(T\star_\Lambda R)\star_\Lambda S=0$. 
 Finally, to obtain the equivalent expression recall from Theorem \ref{thm:biorthogonal} that $R=b\star_\Lambda \check{S}^*$ for $b\in \ell^1(\Lambda).$ Hence by associativity and commutativity of convolutions, $$c=T\star_\Lambda R=b\star_\Lambda (T\star_\Lambda \check{S}^*).$$ It follows from \eqref{eq:fourierseriesofconvolution} that we get $$\F_\sigma^\Lambda(c)=\F_\sigma^\Lambda(b)\F_\sigma^\Lambda(T\star_\Lambda \check{S}^*).$$ We have a known expression for $\F_\sigma^\Lambda(b)$ from Theorem \ref{thm:biorthogonal}, and a known expression for $\F_\sigma^\Lambda(T\star_\Lambda \check{S}^*)$ from Theorem \ref{thm:orthogonality} -- inserting these expressions into the equation above yields the desired result. 
\end{enumerate}
 \end{proof}

The key to the results of this section is Wiener's lemma, used in the proof of Theorem \ref{thm:biorthogonal}. In fact, we may interpret these results as a variation of Wiener's lemma. To see this, recall that $V^2=\overline{\text{span}\{\alpha_\lambda(S):\lambda \in \Lambda\}}=\ell^2(\Lambda)\star_\Lambda S\subset \HS$. Then $\{\alpha_\lambda(S)\}_{\lambda \in \Lambda}$ is a Riesz sequence if and only if the convolution map $D_S:\ell^2(\Lambda)\to V^2$ given by $$D_S(c)= c\star_\Lambda S$$ has a bounded inverse \cite[Thm. 3.6.6]{Christensen:2016}. Corollary \ref{cor:banachisomorphism} therefore says the following: if $S\in \beauty$ and the convolution map $D_S: \ell^2(\Lambda)\to V^2$ has a bounded inverse, then the inverse is given by the convolution $$C_R(T)= R\star_\Lambda T$$ for some $R\in \beauty$. The similarities with Wiener's lemma are evident when we compare this to the following formulation of Wiener's lemma\cite[Thm. 5.18]{Grochenig:2010}:
\begin{quote}
	If $b\in \ell^1(\Z)$ and the convolution map $\ell^2(\Z)\to \ell^2(\Z)$ defined by  $$c\mapsto c\ast_\Z b$$ has a bounded inverse on $\ell^2(\Z)$, then the inverse is given by the convolution map $$c\mapsto c\ast_{\Z} b'$$ for some $b'\in \ell^1(\Z)$.

\end{quote}

\section{Tauberian theorems} \label{sec:tauberian}
In the continuous setting, where one considers functions on $\Rdd$ and the convolutions briefly introduced at the beginning of Section \ref{sec:convolutions}, a version of Wiener's Tauberian theorem for operators was obtained by Kiukas et al. \cite{Kiukas:2012}, building on earlier work by Werner \cite{Werner:1984}. This theorem consists of a long list of equivalent statements for $\SC^p$ and $L^p(\Rdd)$ for $p=1,2,\infty$, and as a starting point for our discussion we state a shortened version for $p=2$ below.

\begin{thm} \label{thm:kiukastauberian}
	Let $S\in \tco$. The following are equivalent.
	\begin{enumerate}
		\item The span of $\{\alpha_z(S)\}_{z\in \Rdd}$ is dense in $\HS$.
		\item The set of zeros of $\F_W(S)$ has Lebesgue measure $0$ in $\Rdd$.
		\item The set of zeros of $\F_\sigma(S\star \check{S}^*)$ has Lebesgue measure $0$ in $\Rdd$.
		\item If $f\star S=0$ for $f\in L^2(\Rdd)$, then $f=0$.
		\item If $T\star S=0$ for $T\in \HS$, then $T=0$.
	\end{enumerate}
\end{thm}

We wish to obtain versions of this theorem when $\Rdd$ is replaced by a lattice $\Lambda,$ functions on $\Rdd$ are replaced by sequences on $\Lambda$ and we still consider operators on $L^2(\Rd)$. In this discrete setting, statements (3) and (4) in Theorem \ref{thm:kiukastauberian} are still equivalent, mutatis mutandis, while the analogues of (1) and (5) can never be true. First we show that the discrete version of statement (1) can never hold. 

\begin{prop} \label{prop:nodensity}
	Let $\Lambda$ be any lattice in $\Rdd$ and let $S\in \HS$. Then the linear span of $\{\alpha_\lambda(S)\}_{\lambda \in \Lambda}$ is not dense in $\HS$.
\end{prop}
\begin{proof}
As we have exploited on several occasions, the Weyl transform is unitary from from $L^2(\Rdd)$ to $\HS$ and sends translations of operators using $\alpha$ to translations of functions. It is therefore sufficient to show that  $\{T_\lambda(\weyl_S)\}_{\lambda \in \Lambda}$ is not dense in $L^2(\Rdd)$, where $\weyl_S$ is the Weyl symbol of $S$.
	Let $c:=\frac{2}{|\Lambda|}$, and define $\Lambda'=c\Z^{2d}$. Consider the lattice $\Lambda \times \Lambda'$ in $\R^{4d}$. Then we have that $|\Lambda \times \Lambda'|=|\Lambda|\cdot c=2>1$. By the density theorem for Gabor systems \cite{Grochenig:2001,Heil:2007,Bekka:2004}, this implies that the system $\{\pi(\lambda,\lambda') \weyl_S\}_{(\lambda,\lambda')\in \Lambda \times \Lambda'}$ cannot be span a dense subset in $L^2(\Rdd)$, so in particular the subsystem $\{\pi(\lambda,0) \weyl_S\}_{(\lambda,0)\in \Lambda \times \Lambda'}=\{T_{\lambda} \weyl_S\}_{\lambda\in \Lambda }$ cannot be complete.
\end{proof}
This implies that we cannot hope to generalize part (5) of Theorem \ref{thm:kiukastauberian} to the discrete setting.
\begin{cor}
	Let $S\in \beauty$. There exists $0\neq T\in \HS$ such that $T\star_\Lambda S=0.$
\end{cor}
\begin{proof}
	To obtain a contradiction, we assume that $T\star_\Lambda S=0\implies T=0$ for $T\in \HS$. As we have seen in \eqref{eq:convolutionasinnerproduct}, $$T\star_\Lambda S(\lambda)=\inner{T}{\alpha_\lambda(\check{S}^*)}_{\HS}.$$ Our assumption is therefore equivalent to $$\inner{T}{\alpha_\lambda(\check{S}^*)}_{\HS}=0 \text{ for all } \lambda \in \Lambda \implies T=0,$$ which implies that the linear span of $\{\alpha_\lambda(\check{S}^*)\}_{\lambda\in \Lambda}$ is dense in $\HS$ -- a contradiction to Proposition \ref{prop:nodensity} applied to $\check{S}^*\in \beauty$.
\end{proof}
Proposition \ref{prop:nodensity} also allows us to construct counterexamples to the associativity of convolutions of three operators. 
\begin{cor} \label{cor:noassociativity}
	Assume that $\{\alpha_\lambda(S)\}_{\lambda \in \Lambda}$ is a Riesz sequence in $\HS$ for $S\in \beauty$. Then there exist $R\in \beauty$ and $T\in \HS$ such that $$(T\star_\Lambda R) \star_\Lambda S\neq T\star_\Lambda (R\star_\Lambda S).$$
\end{cor}
\begin{proof}
	Choose $R\in \beauty$ as in Section \ref{sec:biorthogonal}, i.e. such that $S\star_\Lambda R=\delta_{\lambda,0}$. Then use Proposition \ref{prop:nodensity} to pick $T\in \HS$ that does not belong to  the closed linear span of $\{\alpha_\lambda(S)\}_{\lambda \in \Lambda}$ in $\HS$. We get that
	\begin{equation*}
  T\star_\Lambda (R\star_\Lambda S)=T\star_\Lambda \delta_{\lambda,0}=T.
\end{equation*}
If we assumed associativity, we would get
\begin{equation*}
  T=(T\star_\Lambda R) \star_\Lambda S,
\end{equation*}
where $T\star_\Lambda R\in \ell^2(\Lambda)$ by Proposition \ref{prop:youngschatten}. Hence we could express $T=c\star_\Lambda S$ for $c\in \ell^2(\Lambda)$, which would imply that $T$ belongs to the closed linear span of $\{\alpha_\lambda(S)\}_{\lambda \in \Lambda}$ by \eqref{eq:closedspanconvolution} -- a contradiction.
\end{proof}
On the positive side, we can use the techniques developed in Section \ref{sec:fouriertransforms} to prove the following theorem, which shows that parts (3) and (4) of Theorem \ref{thm:kiukastauberian} have natural analogues for sequences. For Gabor multipliers, Feichtinger was interested in the question of recovering $c$ from $c\star_\Lambda (\varphi\otimes \varphi)$, and the continuity of the mapping $c\star_\Lambda (\varphi\otimes \varphi)\mapsto c$. In this case he proved the equivalence $(1)(i) \iff (1)(iv)$ below \cite[Thm. 5.17]{Feichtinger:2002}, and that this implies the final statement in part $(1)$\cite[Prop. 5.22 and Prop. 5.23]{Feichtinger:2002}. In part $(3)$ we show that any $c\in \ell^1(\Lambda)$ (in particular any finite sequence) can be recovered from $c\star_\Lambda S$ under significantly weaker assumptions on $S$ for a fixed lattice $\Lambda$, but obtain no continuity statement. 

\begin{thm}  \label{thm:bigtauberian}
	Let $S\in \beauty$. 
	\begin{enumerate}
		\item The following are equivalent:
			\begin{enumerate}[(i)]
			\item $\F_\sigma^\Lambda(S\star_\Lambda \check{S}^*)$ has no zeros in $\Rdd/\Lambda^\circ$.
			\item If $c\star_\Lambda S=0$ for $c\in \ell^\infty(\Lambda)$, then $c=0$.
			\item $\beauty\star_\Lambda S$ is dense in $\ell^1(\Lambda).$
			\item $\{\alpha_\lambda S\}_{\lambda \in \Lambda}$ is a Riesz sequence in $\HS$.
			\end{enumerate}
			If any of the statements above holds, $c\in \ell^\infty(\Lambda)$ is recovered from $c\star_\Lambda S$ by $c=(c\star_\Lambda S)\star_\Lambda R$ for some $R\in \beauty.$ In particular, the map $c\star_\Lambda S\mapsto c$ is continuous $\bo \to \ell^\infty(\Lambda)$.
		\item The following are equivalent:
			\begin{enumerate}[(i)]
			\item $\F_\sigma^\Lambda (S\star_\Lambda \check{S}^*)$ is non-zero a.e. in $\Rdd/\Lambda^\circ$.
			\item If $c\star_\Lambda S=0$ for $c\in \ell^2(\Lambda)$, then $c=0$.
			\item $\HS\star_\Lambda S$ is dense in $\ell^2(\Lambda).$
			\end{enumerate}
		\item The following are equivalent:
			\begin{enumerate}[(i)]
			\item The set of zeros of $\F_\sigma^\Lambda(S\star_\Lambda \check{S}^*)$ contains no open subsets in $\Rdd/\Lambda^\circ$.
			\item If $c\star_\Lambda S=0$ for $c\in \ell^1(\Lambda)$, then $c=0$.
			\item $\beast \star_\Lambda S$ is weak*-dense in $\ell^\infty(\Lambda)$.
			\end{enumerate}
	\end{enumerate}
\end{thm}
\begin{proof}
	\begin{enumerate}
		\item The equivalence of $(i)$ and $(iv)$ was the content of Theorem \ref{thm:rieszsequence}. By Corollary \ref{cor:banachisomorphism}, $(iv)$ implies that $c\mapsto c\star_\Lambda S$ is injective, hence $(i)\iff (iv)\implies (ii)$ holds. Then assume that $(ii)$ holds, and let $\dot{z}\in \Rdd/\Lambda^\circ$ -- to show $(i)$, we need to show that $\F_\sigma^\Lambda(S\star_\Lambda \check{S}^*)(\dot{z})\neq 0$, which by Corollary \ref{cor:orthogonalabsolute} is equivalent to showing that there exists some $\lambda^\circ\in \Lambda^\circ$ such that $\F_W(S)(z+\lambda^\circ)\neq 0$. 
 
		Consider the distribution $\delta_{\dot{z}}\in A'(\Rdd/\Lambda^\circ)$ defined by $$\inner{\delta_{\dot{z}}}{f}_{A'(\Rdd/\Lambda^\circ),A(\Rdd/\Lambda^\circ)} =\overline{f(\dot{z})}$$ (recall that our duality brackets are antilinear in the second coordinate), and let $c^{\dot{z}}=\{c^{\dot{z}}(\lambda)\}_{\lambda \in \Lambda}\in \ell^\infty(\Lambda)$ be its symplectic Fourier coefficients, i.e. $\F_\sigma^\Lambda(c^{\dot{z}})=\delta_{\dot{z}}$. We know that $c^{\dot{z}}\star_\Lambda S\in \beast$ is non-zero by $(ii)$, and Proposition \ref{prop:spreadinggm2} gives for any $f\in S_0(\Rdd)$ that
		{\small \begin{align*} 
	\inner{\F_W(c^{\dot{z}}\star_\Lambda S)}{f}_{S_0',S_0}&=\inner{\delta_{\dot{z}} \F_W(S)}{f}_{S_0',S_0} \\
	&= \inner{\delta_{\dot{z}} }{\overline{\F_W(S)}f}_{S_0',S_0} \\
	&= \inner{\delta_{\dot{z}} }{P_{\Lambda^\circ}\left[\overline{\F_W(S)}f \right]}_{A'(\Rdd/\Lambda^\circ),A(\Rdd/\Lambda^\circ)} \quad \text{by } \eqref{eq:periodicextension} \\
	&= P_{\Lambda^\circ}\left[\F_W(S)\overline{f} \right](\dot{z}) \\
	&=  \sum_{\lambda^\circ\in \Lambda^\circ}\F_W(S)(z+\lambda^\circ)\overline{f(z+\lambda^\circ)}.
\end{align*} }
		From this it is clear that if $\F_W(S)(z+\lambda^\circ)=0$ for all $\lambda^\circ \in \Lambda^\circ$, then $\F_W(c^{\dot{z}}\star_\Lambda S)=0$ and hence $c^{\dot{z}}\star_\Lambda S=0$ since $\F_W:\beast\to S_0(\Rdd)$ is an isomorphism, which cannot hold by $(ii)$.
		 
		 Before we prove $(ii)\iff (iii)$, note that $(i)$ is unchanged when $S\mapsto \check{S}^*$ by commutativity of the convolutions. Since $(i)\iff (ii)$, this means that $(ii)$ is equivalent to 
		 \begin{enumerate}[(ii')]
		 \item If $c\star_\Lambda \check{S}^*=0$ for $c\in \ell^\infty(\Lambda)$, then $c=0$.
		 \end{enumerate}
		  To prove the equivalence of $(ii')$ and $(iii)$, we will prove that the map $D_{\check{S}^*}:\ell^\infty(\Lambda)\to \beauty'$ given by $D_{\check{S}^*}(c)= c\star_\Lambda \check{S}^*$ is the Banach space adjoint of $C_S:\beauty\to \ell^1(\Lambda)$ given by $C_S(T)=T\star_\Lambda S.$ This amounts to proving that
		\begin{equation*} 
  \inner{D_{\check{S}^*}(c)}{T}_{\beast,\beauty}=\inner{c}{C_S(T)}_{\ell^\infty(\Lambda),\ell^1(\Lambda)} \quad \text{ for } T\in \beauty, c\in \ell^\infty(\Lambda).
\end{equation*}
  By writing out the definitions of $D_{\check{S}^*}$ and $C_S$, we see that we need to show that 
  \begin{equation*}
  \inner{c\star_\Lambda \check{S}^*}{T}_{\beast,\beauty}=\inner{c}{T\star_\Lambda S}_{\ell^\infty,\ell^1} \quad \text{ for } T\in \beauty, c\in \ell^\infty(\Lambda),
\end{equation*}
which is simply the definition of $c\star_\Lambda \check{S}^*$ when $c\in \ell^\infty(\Lambda)$ from \eqref{eq:dualconvolutions}, hence true. Since a bounded linear operator between Banach spaces has dense range if and only if its Banach space adjoint is injective (see \cite[Corollary to Thm. 4.12]{Rudin:2006}, part (b)), this implies that $(ii')$ is equivalent to $(iii)$. Finally, Corollary \ref{cor:banachisomorphism} implies the final statement that $c=(c\star_\Lambda S)\star_\Lambda R$. 
\item The equivalence $(ii)\iff(iii)$ is proved as above . Assume that $(i)$ holds, and that $c\star_\Lambda S=0$ for some $c\in \ell^2(\Lambda)$. By associativity of convolutions, $$c\ast_\Lambda (S\star_\Lambda \check{S}^*)=0.$$ Applying $\F_\sigma^\Lambda$ to this, we find using \eqref{eq:fourierseriesofconvolution} that $$\F_\sigma^\Lambda(c)\F_\sigma^\Lambda (S\star_\Lambda \check{S}^*)=0.$$ By $(i)$ this implies that $\F_\sigma^\Lambda(c)=0$ in $L^2(\Rdd/\Lambda^\circ)$, hence $c=0$.

Then assume that $(i)$ does not hold, i.e. there is a subset $U\subset \Rdd/\Lambda^\circ$ of positive measure where $\F_\sigma^\Lambda(S\star_\Lambda \check{S}^*)$ vanishes. Pick $c\in \ell^2(\Lambda)$ such that $\F_\sigma^\Lambda(c)=\chi_U,$ where $\chi_U$ is the characteristic function of $U$, which is possible since $\F_\sigma^\Lambda:\ell^2(\Lambda)\to L^2(\Rdd/\Lambda^\circ)$ is unitary and so in particular onto. Then by Proposition \ref{prop:spreadinggm2}, for $f\in S_0(\Rdd)$, 
{\small \begin{align*} 
	\inner{\F_W(c\star_\Lambda S)}{f}_{S_0',S_0}&=\inner{\chi_U \F_W(S)}{f}_{S_0',S_0} \\
	&= \inner{\chi_U }{\overline{\F_W(S)}f}_{S_0',S_0} \\
	&= \inner{\chi_U }{P_{\Lambda^\circ}\left[\overline{\F_W(S)}f \right]}_{A'(\Rdd/\Lambda^\circ),A(\Rdd/\Lambda^\circ)} \quad \text{by } \eqref{eq:periodicextension} \\
	&=  \int_{\Rdd/\Lambda^\circ} \chi_U(\dot{z}) \sum_{\lambda^\circ\in \Lambda^\circ}\F_W(S)(z+\lambda^\circ)\overline{f(z+\lambda^\circ)}d\dot{z} \\
	&=0.
\end{align*} }
To see why the last integral is zero, note first that if $\dot{z}\notin U$, then $\chi_U(\dot{z})=0.$  If $\dot{z}\in U$, then we use that by Corollary \ref{cor:orthogonalabsolute},
\begin{equation*}
   \F_\sigma^\Lambda (S\star_\Lambda \check{S^*})(\dot{z})=\frac{1}{|\Lambda|}\sum_{\lambda^\circ \in \Lambda^\circ} |\F_W(S)(z+\lambda^\circ)|^2 \text{ for any } z\in \Rdd.
\end{equation*}
Hence the assumption $\F_\sigma^\Lambda (S\star_\Lambda \check{S^*})(\dot{z})=0$ for $\dot{z}\in U$ implies that $\F_W(S)(z+\lambda^\circ)=0$ for any $\lambda^\circ\in \Lambda^\circ$ when $\dot{z}\in U$. In conclusion we have shown that the integrand above is zero, hence the integral is zero. This means that $\F_W(c\star_\Lambda S)=0$, so $c\star_\Lambda S=0$, contradicting $(ii)$ since $c\neq 0$.
\item  Assume that $(i)$ holds, and that $c\star_\Lambda S=0$ for some $c\in \ell^1(\Lambda)$. By associativity, we also have that $c\star_\Lambda (S\star_\Lambda \check{S}^*)=0$, and by applying $\F_\sigma^\Lambda$ we get from \eqref{eq:fourierseriesofconvolution}
\begin{equation*}
  \F_\sigma^\Lambda(c)(\dot{z})\F_\sigma^\Lambda(S\star_\Lambda \check{S}^*)(\dot{z})=0 \quad \text{ for any } \dot{z}\in \Rdd/\Lambda^\circ.
\end{equation*}
Since $c\in \ell^1(\Lambda)$, $\F_\sigma^\Lambda(c)$ is a continuous function. So if $c\neq 0$, there must exist an open subset $U\subset \Rdd/\Lambda^\circ$ such that $\F_\sigma^\Lambda(c)(\dot{z})\neq 0$ for $\dot{z}\in U$. But the equation above then gives that $\F_\sigma(S\star \check{S}^*)(\dot{z})=0$ for $\dot{z}\in U$; a contradiction to $(i)$. Hence $c=0,$ and $(ii)$ holds.
 Then assume that $(ii)$ holds, and assume that there is an open set $U\subset \Rdd/\Lambda^\circ$ such that $\F_\sigma^\Lambda(S\star_\Lambda \check{S}^*)(\dot{z})=0$ for any $\dot{z}\in U.$ By Theorem \ref{thm:orthogonality}, this means that
 \begin{equation*}
  \sum_{\lambda^\circ\in \Lambda^\circ} |\F_W(S)(z+\lambda^\circ)|^2=0 \quad \text{ when } \dot{z}\in U,
\end{equation*}
which is clearly equivalent to 
\begin{equation*}
  \F_W(S)(z)=0 \quad \text{ whenever } \dot{z}\in U.
\end{equation*}

Then find some non-zero $c\in \ell^1(\Lambda)$ such that $\F_\sigma^\Lambda(c)$ vanishes outside $U$, which is possible by \cite[Remark 5.1.4]{Reiter:2000}. Using Proposition \ref{prop:spreadinggm}, we have 
\begin{equation*}
  \F_W(c\star_\Lambda S)(z)=\F_\sigma^\Lambda(c)(\dot{z})\F_W(S)(z)\quad \text{ for } z\in \Rdd.
\end{equation*}
If $\dot{z}\notin U$, then $\F_\sigma^\Lambda(c)(\dot{z})=0$ by construction of $c$. Similarly, if $\dot{z}\in U$, then we saw that $\F_W(S)(z)=0$. Hence $\F_W(c\star_\Lambda S)(z)=0$ for any $z\in \Rdd$, which implies that $c\star_\Lambda S=0$. But $c\neq 0$, so this is impossible when we assume $(ii)$, so there cannot exist an open subset $U\subset \Rdd/\Lambda^\circ$ such that $\F_\sigma^\Lambda(c)(\dot{z})\neq 0$ for $\dot{z}\in U$.

The equivalence $(ii)\iff (iii)$ is proved as in part (1), with some minor modifications. We note that $(i)$ is unchanged when $S\mapsto \check{S}^*$, so as $(i)\iff (ii)$ we have that $(ii)$ is equivalent to 
		 \begin{enumerate}[(ii')]
		 \item If $c\star_\Lambda \check{S}^*=0$ for $c\in \ell^1(\Lambda)$, then $c=0$.
		 \end{enumerate}
By simply writing out the definitions, one sees using \eqref{eq:dualconvolutions3} that the map $C_S:\beast\to \ell^\infty(\Lambda)$ given by $C_S(T)=T\star_\Lambda S$   is the Banach space adjoint of $D_{\check{S}^*}:\ell^1(\Lambda)\to \beauty$ given by $D_{\check{S}^*}(c)= c\star_\Lambda \check{S}^*$. The equivalence $(ii')\iff (iii)$ therefore follows from part (c) of \cite[Corollary of Thm. 4.12]{Rudin:2006}: a bounded linear operator between Banach spaces is injective if and only if the range of its adjoint is weak*-dense.

	\end{enumerate}
\end{proof}
Let us rewrite the statements of the theorem in the case that $S$ is a rank-one operator $S=\varphi\otimes \varphi$ for $\varphi\in S_0(\Rd)$. By \eqref{eq:tworankone} we find that 
\begin{equation*}
  S\star_\Lambda \check{S}^*(\lambda)=|V_{\varphi}\varphi (\lambda)|^2,
\end{equation*}
and by \eqref{eq:gabormultiplier} $c\star_\Lambda S$ is the Gabor multiplier 
\begin{equation*}
  c\star_\Lambda (\varphi\otimes \varphi)\psi =\sum_{\lambda \in \Lambda} c(\lambda)V_{\varphi}\psi(\lambda)\pi(\lambda)\varphi.
\end{equation*}
Hence the equivalences $(i)\iff (ii)$ provides a characterization using the symplectic Fourier series of $V_\varphi \varphi\vert_{\Lambda}$ of when the symbol $c$ of a Gabor multiplier is uniquely determined.

\subsection{Underspread operators and a Wiener division lemma}
For motivation, recall Wiener's division lemma \cite[Lem. 1.4.2]{Reiter:2000}: if $f,g\in L^1(\Rdd)$ satisfy that $\hat{f}$ has compact support ($\hat{f}$ is the usual Fourier transform on $\Rdd$) and $\hat{g}$ does not vanish on $\text{supp}(\hat{f}),$ then $$f=f\ast h\ast g$$ for some $h\in L^1(\Rdd)$ satisfying $\hat{h}(z)=\frac{1}{\hat{g}(z)}$ for $z\in \text{supp}(\hat{f})$. The next result is a version of this statement for the convolutions and Fourier transforms of operators and sequences. At the level of Weyl symbols, this result is due to Gr\"ochenig and Pauwels \cite{Grochenig:2014} (see also the thesis of Pauwels \cite{Pauwels:2011}) using different techniques. We choose to include a proof using the techniques of this paper to show how the the statement fits our formalism. Note that apart from the function $g$ -- introduced to ensure $A\in \beauty$ -- Theorem \ref{thm:underspread} is obtained by replacing the convolutions and Fourier transforms in Wiener's division lemma by the convolutions and Fourier transforms of sequences and operators. 

\begin{rem}
	If $\Lambda^\circ=A\Z^{2d}$, we will pick the fundamental domain $\square_{\Lambda^\circ}=A[-\frac{1}{2},\frac{1}{2})^{2d}$ which means that any $z\in \Rdd$ can be written as $z=z_0+\lambda^\circ$ for $z_0\in \square_{\Lambda^\circ}, \lambda^\circ \in \Lambda^\circ$ in a unique way. This choice of fundamental domain implies that  $(1-\epsilon)\square_{\Lambda^\circ}=A[-\frac{1}{2}+\frac{\epsilon}{2},\frac{1}{2}-\frac{\epsilon}{2})^{2d}$, so we may find $g$ in the statement below by \cite[Prop. 2.26]{Lee:2003}.
\end{rem}

\begin{thm} \label{thm:underspread}
	Assume that $S\in \beauty$ satisfies $\text{supp}(F_W(S))\subset  (1-\epsilon)\square_{\Lambda^\circ}$ for some $0<\epsilon < 1/2$. Pick $g\in C^\infty_c(\Rdd)$ such that $g\vert_{(1-\epsilon)\square_{\Lambda^\circ}}\equiv 1$ and $\text{supp}(g)\subset \square_{\Lambda^\circ}$. If $T\in \beauty$ satisfies $\F_W(T)(z)\neq 0$ for $z\in \text{supp}(g)$, then 
	\begin{equation*}
  S=(S\star_\Lambda T)\star_\Lambda A,
\end{equation*}
where $A\in \beauty$ is given by $\F_W(A)=\frac{g}{\F_W(T)}$.
\end{thm}

\begin{proof}
	We first show that $A\in \beauty$ by showing $\F_W(A)\in S_0(\Rdd)$. The Wiener-L\'{e}vy theorem \cite[Thm. 1.3.1]{Reiter:2000} gives $h\in L^1(\Rdd)$ such that $\hat{h}(z)=1/\F_W(T)(z)$ for $z\in \text{supp}(g),$ where $\hat{}$ denotes the usual Fourier transform. Therefore $\F_W(A)=g\cdot \hat{h}$, which belongs to $S_0(\Rdd)$ by \cite[Prop. 12.1.7]{Grochenig:2001}.
	
	To show that $S=(S\star_\Lambda T)\star_\Lambda A$, we will show that their Fourier-Wigner transforms are equal. Using Proposition \ref{prop:spreadinggm} and Theorem \ref{thm:orthogonality} we find that {\footnotesize 
	\begin{align*}
  \F_W((S\star_\Lambda T)\star_\Lambda A)(z)&= \F_\sigma^\Lambda (S\star_\Lambda T)(\dot{z})\F_W(A)(z) \\
  &= \F_W(A)(z) \sum_{\lambda^\circ \in \Lambda^\circ} \F_W(S)(z+\lambda^\circ)\F_W(T)(z+\lambda^\circ).
\end{align*}}
To show that this equals $\F_W(S)$, we consider three cases.
\begin{itemize}
	\item If $z\in (1-\epsilon)\square_{\Lambda^\circ}$, then $g(z)=1$ by construction and {\footnotesize
\begin{align*}
  \F_W(A)(z) \sum_{\lambda^\circ \in \Lambda^\circ} \F_W(S)(z+\lambda^\circ)\F_W(T)(z+\lambda^\circ)&=\F_W(A)(z)\F_W(S)(z)\F_W(T)(z) \\
  &= \frac{g(z)}{\F_W(T)(z)}\F_W(S)(z)\F_W(T)(z)\\
  &=\F_W(S)(z),
\end{align*}}
where we used that the only summand contributing to the sum is $\lambda^\circ=0$ since $\text{supp}(\F_W(S))\subset \square_{\Lambda^\circ}$ and $z\in \square_{\Lambda^\circ}$ and $\square_{\Lambda^\circ}$ is a fundamental domain.
\item If $z\in \square_{\Lambda^\circ}\setminus (1-\epsilon)\square_{\Lambda^\circ}$, then $\F_W(S)(z)=0$ and the same argument as above gives {\footnotesize
\begin{align*}
  \F_W(A)(z) \sum_{\lambda^\circ \in \Lambda^\circ} \F_W(S)(z+\lambda^\circ)\F_W(T)(z+\lambda^\circ)&=\F_W(A)(z)\overbrace{\F_W(S)(z)}^{0}\F_W(T)(z) \\
  &= 0.
\end{align*} }
\item If $z\notin (1-\epsilon)\square_{\Lambda^\circ}$, then $\F_W(S)(z)=0$ since $\text{supp}(\F_W(S))\subset \square_{\Lambda^\circ}$ and $\F_W((S\star_\Lambda T)\star_\Lambda A)(z)=0$ since $\F_W(A)(z)=\frac{g(z)}{\F_W(T)(z)}=0$ as $\text{supp}(g)\subset \square_{\Lambda^\circ}$.
\end{itemize}
\end{proof}

	A similar argument using duality brackets shows that essentially the same result even holds for $S\in \beast$. 
\begin{thm} \label{thm:dualunderspread}
	Assume that $S\in \beast$ satisfies $\text{supp}(F_W(S))\subset  (1-2\epsilon)\square_{\Lambda^\circ}$ for some $0<\epsilon < 1/2$. Pick $g\in C^\infty_c(\Rdd)$ such that $g\vert_{(1-\epsilon)\square_{\Lambda^\circ}}\equiv 1$ and $\text{supp}(g)\subset \square_{\Lambda^\circ}$. If $T\in \beauty$ satisfies  $\F_W(T)(z)\neq 0$ for $z\in \text{supp}(g)$, then 
	\begin{equation*}
  S=(S\star_\Lambda T)\star_\Lambda A,
\end{equation*}
where $A\in \beauty$ is given by $\F_W(A)=\frac{g}{\F_W(T)}$.
\end{thm}

\begin{proof}
We have already seen that $A\in \beauty$. Let $f\in S_0(\Rdd)$. Then  {\footnotesize 
	\begin{align*}
  \inner{\F_W\left[(S\star_\Lambda T)\star_\Lambda A)\right]}{f}_{S_0',S_0}&=   \inner{(S\star_\Lambda T)\star_\Lambda A)}{\rho(f)}_{\beast,\beauty} \quad \text{ by \eqref{eq:fwdual}} \\
  &= \inner{S\star_\Lambda T}{\rho(f)\star_\Lambda \check{A}^*}_{\ell^\infty,\ell^1} \quad \text{ by \eqref{eq:dualconvolutions}}  \\
&= \inner{S}{(\rho(f)\star_\Lambda \check{A}^*)\star_\Lambda \check{T}^*}_{\beast,\beauty} \quad \text{ by \eqref{eq:dualconvolutions3}} \\
&= \inner{\F_W(S)}{\F_W\left[(\rho(f)\star_\Lambda \check{A}^*)\star_\Lambda \check{T}^* \right]}_{S_0',S_0} \quad \text{ by \eqref{eq:fwdual}} \\
&= \inner{\F_W(S)}{b\cdot \F_W\left[(\rho(f)\star_\Lambda \check{A}^*)\star_\Lambda \check{T}^* \right]}_{S_0',S_0} 
\end{align*} }
In the last line we multiplied the right hand side by a bump function $b\in C_c^\infty(\Rdd)$ such that $b\vert_{(1-2\epsilon)\square_{\Lambda^\circ}}\equiv 1$ and $\text{supp}(b)\subset (1-\epsilon)\square_{\Lambda^\circ}$ -- this does not change anything by the assumptions on the supports of $\F_W(S)$ and $b$. We find using Theorem \ref{thm:orthogonality} and Proposition \ref{prop:spreadinggm} that 
\begin{align*}
  b\cdot \F_W\left[(\rho(f)\star_\Lambda \check{A}^*)\star_\Lambda \check{T}^* \right]&=b\cdot \F_\sigma^\Lambda (\rho(f)\star_\Lambda \check{A}^*)\cdot  \overline{\F_W(T)} \\
  &=b \cdot \overline{\F_W(T)} P_{\Lambda^\circ}(f\overline{\F_W(A)}) \quad \text{ by \eqref{eq:fwcheckadjoint}}.
\end{align*}
We claim that this last function equals $b\cdot f$: if $z\notin (1-\epsilon)\square_{\Lambda^\circ}$, then $b(z)=0$, so $b(z)f(z)=0$ and $$b(z) \cdot \overline{\F_W(T)(z)} P_{\Lambda^\circ}(f\overline{\F_W(A)})(\dot{z})=0.$$
If $z\in (1-\epsilon)\square_{\Lambda^\circ}$, then $g(z)=1$ and {\footnotesize
\begin{align*}
  b(z)\overline{\F_W(T)(z)}P_{\Lambda^\circ}(f\overline{\F_W(A)})(\dot{z})&= b(z)\overline{\F_W(T)(z)}\sum_{\lambda^\circ \in \Lambda^\circ} f(z+\lambda^\circ) \overline{\F_W(A)}(z+\lambda^\circ) \\
  &= b(z)\overline{\F_W(T)(z)}f(z)\overline{\F_W(A)(z)} \\
  &= b(z)\overline{\F_W(T)(z)}f(z)\frac{\overline{g(z)}}{\overline{\F_W(T)(z)}} \\
  &= b(z)f(z).
\end{align*} }
since $\F_W(A)$ vanishes outside of $\square_{\Lambda^\circ}$ by construction. Hence we have shown that 
\begin{align*}
  \inner{\F_W\left[(S\star_\Lambda T)\star_\Lambda A)\right]}{f}_{S_0',S_0}&=\inner{\F_W(S)}{b\cdot f}_{S_0',S_0}  \\
  &= \inner{\F_W(S)}{f}_{S_0',S_0}
\end{align*}
for any $f\in S_0(\Rdd)$, which implies the result. 

\end{proof}

Operators $S$ such that $\text{supp}(\F_W(S))\subset [-\frac{a}{2},\frac{a}{2}]^d \times [-\frac{b}{2},\frac{b}{2}]^d$ where $ab\leq 1$ are called \textit{underspread}, and provide realistic models of communication channels \cite{Kozek:2006,Kozek:1997,Strohmer:2006,Grochenig:2014,Dorfler:2010}. We immediately obtain the following consequence.

\begin{cor}
	Any underspread operator $S\in \beast$ can be expressed as a convolution $T=c\star_\Lambda A$ with $c\in \ell^\infty(\Lambda)$ and $A\in \beauty$ for a sufficiently dense lattice $\Lambda$. In particular, $S$ is bounded on $L^2(\Rd)$.
\end{cor}

  It is known (see \cite{Dorfler:2010}) that for an operator $S$ to be well-approximated by Gabor multipliers -- i.e. operators $c\star_\Lambda (\psi\otimes \psi)$ for $\psi\in L^2(\Rd)$ -- $S$ should be underspread. The result above shows that any underspread operator $S$ is given precisely by a convolution $S=c\star_\Lambda A$ if we allow $A$ to be any operator in $\beauty$, not just a rank-one operator. In fact, $A$ as constructed in the theorem will never be a rank-one operator, since $\F_W(A)$ has compact support -- this is not possible for rank-one operators \cite{Janssen:1998a}. If $S$ satisfies $S\in \SC^p$ in addition to the assumptions of Theorem \ref{thm:dualunderspread}, then $c=S\star T\in \ell^p(\Lambda)$ by Proposition \ref{prop:youngschatten}. Hence the $p$-summability of $c$ in $S=c\star_\Lambda A$ reflects the fact that $S\in \SC^p$. 

Theorem \ref{thm:dualunderspread} also implies that underspread operators $S$ are determined by the sequence $S\star_\Lambda T$ when $T\in \beauty$ is chosen appropriately. This was a major motivation for \cite{Grochenig:2014}, since when $T$ is a rank-one operator $T=\varphi\otimes \varphi$, the sequence $S\star_\Lambda \check{T}$ is the diagonal of the so-called channel matrix of $S$ with respect to $\varphi$ -- see \cite{Grochenig:2014,Pauwels:2011} for a thorough discussion and motivation of these concepts. Finally, note that Theorem \ref{thm:dualunderspread} gives a (partial) discrete analogue of part (5) of Theorem \ref{thm:kiukastauberian}. 

\section*{Acknowledgements} 
We thank Franz Luef for insightful feedback on various drafts of this paper. We also thank Markus Faulhuber for helpful discussions and suggestions, particularly concerning Proposition \ref{prop:nodensity}.
\bibliographystyle{plain}         %

\end{document}